\documentclass{amsart}
\usepackage[all]{xy}

\newtheorem{theorem}{Theorem}[section]

\theoremstyle{definition}
\newtheorem{definition}[theorem]{Definition}
\newtheorem{example}[theorem]{Example}
\newtheorem{prop}[theorem]{Proposition}
\newtheorem{corollary}[theorem]{Corollary}
\newtheorem{claim}[theorem]{Claim}

\theoremstyle{remark}
\newtheorem{remark}[theorem]{Remark}

\numberwithin{equation}{section}

\newcommand{\vect}[1]{\boldsymbol{#1}}
\newcommand{\C}{\mathbb{C}}
\newcommand{\R}{\mathbb{R}}

\newcommand{\Z}{\mathbb{Z}}
\newcommand{\rank}{\operatorname{rank}}

\renewcommand{\Im}{\operatorname{Im}}

\renewcommand{\phi}{\varphi}

\renewcommand{\v}{\vect{v}}
\newcommand{\w}{\vect{w}}
\newcommand{\e}{\vect{e}}
\newcommand{\n}{\vect{n}}
\newcommand{\x}{\vect{x}}
\newcommand{\y}{\vect{y}}
\newcommand{\q}{\vect{q}}
\renewcommand{\u}{\vect{u}}
\renewcommand{\l}{\vect{l}}

\newcommand{\zv}{\vect{0}}

\newcommand{\mycomment}[1]{}
\newcommand{\ve}{\varepsilon}
\newcommand{\Aut}{\mathrm{Aut}}
\newcommand{\Diff}{\mathrm{Diff}}
\newcommand{\ord}{\mathrm{ord}}
\newcommand{\Hom}{\mathrm{Hom}}
\newcommand{\T}{{}^t\!}
\usepackage{amssymb}
\usepackage{amsmath}
\usepackage{amsfonts}
\usepackage{amsthm}
\usepackage[active]{srcltx}

\usepackage[pdftex]{graphicx}

\begin{document}

\title[Classification of torus bundles over surfaces]{Classification of orientable torus bundles over closed orientable surfaces}

\author[N. Kasuya]{Naohiko Kasuya}
\address{Department of Mathematics, Faculty of Science, Hokkaido University, 
North 10, West 8, Kita-ku, Sapporo, Hokkaido 060-0810, Japan}
\email{{\tt nkasuya@math.sci.hokudai.ac.jp}}
\thanks{The first author is partially supported by JSPS KAKENHI 21K13797. }

\author[I. Noda]{Issei Noda}
\address{Department of Mathematics, Faculty of Science, Hokkaido University, 
North 10, West 8, Kita-ku, Sapporo, Hokkaido 060-0810, Japan}
\email{{\tt noda.issei.l1@elms.hokudai.ac.jp}}
\thanks{}


\date{December 14, 2025.}

\dedicatory{Dedicated to Professor Takashi Tsuboi on the occasion of his $70$th birthday}


\begin{abstract}
Let $g$ be a non-negative integer, 
$\Sigma _g$ a closed orientable surface of genus $g$, and $\mathcal{M}_g$ its mapping class group. 
We classify all the group homomorphisms $\pi _1(\Sigma _g)\to G$ 
up to the action of $\mathcal{M}_g$ on $\pi _1(\Sigma _g)$ in the following cases; 
(1) $G=PSL(2;\Z)$, (2) $G=SL(2;\Z)$. 
As an application of the case (2), we completely classify orientable $T^2$-bundles over closed orientable surfaces up to bundle isomorphism. In particular, we show that any orientable $T^2$-bundle over $\Sigma _g$ with $g\geq 1$ is isomorphic to the fiber connected sum of $g$ pieces of $T^2$-bundles over $T^2$. 
Moreover, the classification result in the case (1) can be generalized into the case where $G$ is the free product of a finite number of finite cyclic groups. We also apply it to an extension problem of maps from a closed surface to a connected sum of lens spaces. 
\end{abstract}

\maketitle


\section{Introduction}
Let $g$ be a non-negative integer, $\Sigma _g$ a closed orientable surface of genus $g$, 
and $\Diff_+(\Sigma _g)$ and $\mathcal{M}_g$ denote its orientation preserving diffeomorphism group and mapping class group, respectively. 
The aim of this paper is to classify orientable $T^2$-bundles over $\Sigma _g$ up to bundle isomorphisms. 

Orientable $T^2$-bundles appear in various scenes of geometry as important examples. 
For example, elliptic bundles over a complex curve like (some kind of) Hopf surfaces and primary Kodaira surfaces have been extensively studied as compact complex surfaces, 
and they are topologically nothing but orientable $T^2$-bundles over closed orientable surfaces. 
In \cite{Th76}, Thurston proved that a primary Kodaira surface admits non-K\"{a}hler symplectic structures  
when regarded as a smooth real $4$-dimensional manifold. 
This example is called a Kodaira--Thurston manifold. 
In the same paper, he also gave a necessary and sufficient condition 
for a surface bundle over a surface to admit a compatible symplectic structure. 
The result was later refined by Geiges \cite{Ge92} and Walczak \cite{Wa05} in the case of $T^2$-bundles (see also \cite{Na05}). 
In the $4$-dimensional topology, Seifert fibered $4$-manifolds, the analogue of Seifert fibered spaces in dimension $4$, have been classified (\cite{Zi69, Ue90, Ue91, Ke08}), where orientable $T^2$-bundles are treated as the case without multiple fibers. 
Zieschang's result \cite{Zi69}, the starting point of the classification, says that 
the isomorphism class of an orientable $T^2$-bundle over $\Sigma_g$ with $g\geq 2$ is determined 
only by the fundamental group of the total space (see also \cite{Ik88, Hi02}). 
This is an important result, but is not worth being called the classification of bundle isomorphism classes. 
For, it is very difficult to distinguish the isomorphism classes of given two groups in general. 
Therefore, it is important to approach from the front 
to the classification of orientable $T^2$-bundles in terms of their monodromies and Euler classes. 

In general, classification of smooth $F$-bundles up to bundle isomorphisms is one of the most fundamental problems in topology, but at the same time, a very difficult problem. 
For, the classification problem is reduced to that of homotopy classes of continuous maps from the base space to the classifying space $\mathrm{BDiff}(F)$, whose topology is usually hard to analyze. 
In our case, however, the situation is not so bad. 
Indeed, the oriented diffeomorphism group $\mathrm{Diff}_+(T^2)$, which is the structure group of an orientable $T^2$-bundle, is known to be homotopy equivalent to the affine transformation group $\mathrm{Aff}_+(T^2)$, and in particular, we have 
\[ \pi _0(\mathrm{Diff}_+(T^2))\cong SL(2;\Z), \; \pi _1(\mathrm{Diff}_+(T^2))\cong \Z^2, \; \pi _i(\mathrm{Diff}_+(T^2))=0 \; (i\geq 2). \] 
Then we can show that the isomorphism class of an orientable $T^2$-bundle over $\Sigma _g$ 
is determined only by a group homomorphism $\pi _1(\Sigma _g)\to SL(2; \Z)$ 
called the monodromy and a local coefficient cohomology class in $H^2(\Sigma _g; \{\pi _1(\Sigma _g)\})$ 
called the Euler class (Proposition~\ref{isom}). 
They are respectively represented by 
$A_1, B_1, \ldots , A_g, B_g\in SL(2;\Z)$ with $[A_1, B_1]\cdots [A_g, B_g]=E_2$ and $(m, n)\in \Z^2$, 
so we may denote the bundle by $M(A_1, B_1, \ldots , A_g, B_g; m, n)$ (Proposition~\ref{model}). 
Once the monodromy is fixed, we can easily deal with the Euler class, 
so what we really have to do is to classify the group homomorphisms $\pi _1(\Sigma _g)\to SL(2; \Z)$ 
representing the monodromies of $T^2$-bundles up to the action of 
$\mathcal{M}_g$ on $\pi _1(\Sigma _g)$ and the conjugate action of $SL(2;\Z)$ on itself. 

When the base is $S^2$ ($g=0$), the monodromy is trivial and hence only the Euler class is valid. 
Consequently, the set of isomorphism classes of $T^2$-bundles is $\Z_{\geq 0}$ (\S~\ref{genus 0}). 
When the base is $T^2$ ($g=1$), the classification has been settled by Sakamoto--Fukuhara \cite{SF83} (Theorem~\ref{SF}). 
On the other hand, when $g\geq 2$, the classification of monodromies has not yet been established 
though Zieschang's result stated above is known. 

In this sense, the goal of this paper is to classify all the group homomorphisms 
$$\mathrm{Hom}\left(\pi _1(\Sigma _g), SL(2; \Z)\right)=\{ \widetilde{\rho} \colon \pi _1(\Sigma _g)\to SL(2;\Z) \mid \widetilde{\rho} \text{ : homomorphism}\}$$ up to the action of $\mathcal{M}_g$ when $g\geq 2$. 
Instead of trying to classify them directly, our first approach to this problem is to focus on the homomorphism 
$p \circ \widetilde{\rho}\colon \pi _1(\Sigma _g)\to PSL(2; \Z)$ obtained from the monodromy homomorphism $\widetilde{\rho}$ by taking the composition with the quotient map $$p \colon SL(2; \Z)\to PSL(2;\Z)=SL(2;\Z)/\{\pm E_2\}.$$ 
Namely, our immediate goal is the classification of $\mathrm{Hom}\left(\pi _1(\Sigma _g), PSL(2; \Z)\right)$. 
Now we regard $\Sigma _g$ as the boundary of a genus-$g$ handlebody $V_g$ embedded in $\R^3$. 
This allows us to equip both $V_g$ and its boundary $\Sigma_g = \partial V_g$ with the canonical orientations induced by the standard orientation on $\R^3$.
Let $\{\alpha _i,  \beta _i\}_{1\leq i\leq g}$ be the canonical generators of  $\pi _1(\Sigma _g)$, that is, $\alpha _i$ and $\beta _i$ are meridian and longitude, respectively, for each $i$ with $1\leq i\leq g$ 
(see Figure~\ref{fig:generators}). 
Then we have \[ \pi _1(\Sigma _g)=\langle \alpha _1, \beta _1, \ldots, \alpha _g, \beta _g \mid [\alpha _1, \beta _1]\cdots [\alpha _g, \beta _g] \rangle . \] 
Now we are ready to state our first theorem, which will be a breakthrough to the classification. 
\begin{figure}
\begin{center}
\includegraphics[width=10cm, pagebox=cropbox,clip]{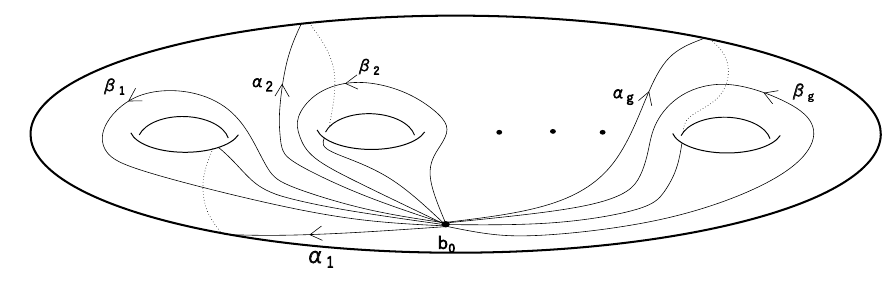}
\caption{Generators of $\pi _1(\Sigma _g)$}
\label{fig:generators}
\end{center}
\end{figure}

\vspace{8pt}

\begin{theorem}~\label{psl}
For any homomorphism $\rho \colon \pi _1(\Sigma _g)\to PSL(2; \Z)$, 
there exists a mapping class $f\in \mathcal{M}_g$ such that 
$(\rho \circ f_{\ast })(\alpha _i)=e$ for each $i$ with $1\leq i\leq g$, where $e$ denotes the identity element of $PSL(2; \Z)$. 
\end{theorem}

The following is an immediate corollary to Theorem~\ref{psl}. 

\vspace{8pt}

\begin{corollary}~\label{sl}
For any homomorphism $\widetilde{\rho }\colon \pi _1(\Sigma _g)\to SL(2; \Z)$, 
there exists a mapping class $f\in \mathcal{M}_g$ such that 
$(\widetilde{\rho }\circ f_{\ast })(\alpha _i)=\pm E_2$ for each $i$ with $1\leq i\leq g$. 
\end{corollary}

In the proof of Theorem~\ref{psl}, the tool called ``charts" plays an essential role. 
This was originally introduced by Kamada \cite{Ka07} for the study of surface-knots, 
but is a very useful tool for describing $G$-representations of surface groups for a given finitely presented group $G$. 
It also works very well in our case, since $PSL(2;\Z)$ is isomorphic to the free product $\Z_2\ast \Z_3$. 
Then, based on Theorem~\ref{psl} and combining several known results of group theory and mapping class group of handlebody $V_g$, we can obtain the classification of $\mathrm{Hom}\left(\pi _1(\Sigma _g), PSL(2; \Z)\right)$ as follows. 
First, any subgroup $H$ of $PSL(2; \Z)\cong \Z_2\ast \Z_3$ is isomorphic to 
the free product of finite number of copies of $\Z$, $\Z_3$ and $\Z_2$ 
by Kurosh's subgroup theorem (Theorem~\ref{Kurosh}). 
To be more precise, there exist subgroups $H_j\subset H$ ($1\leq j\leq m$) such that 
\[H_1\cong \cdots \cong H_k\cong \Z, \;\; 
H_{k+1}\cong \cdots \cong H_l\cong \Z_3, \;\;  
H_{l+1}\cong \cdots \cong H_m\cong \Z_2 \] 
and 
\[ H=(H_1\ast \cdots \ast H_k)
\ast (H_{k+1}\ast \cdots \ast H_l)
\ast (H_{l+1} \ast \cdots \ast H_m), \] 
where $k, l$ and $m$ are integers satisfying $0\leq k\leq l\leq m$. 
Then our classification theorem is as follows. 

\vspace{8pt}

\begin{theorem}~\label{hom-psl}
For any homomorphism $\rho\colon \pi _1(\Sigma _g)\to PSL(2; \Z)\cong \Z_2\ast \Z_3$, the subgroup $\mathrm{Im}(\rho )\subset \Z_2\ast \Z_3$ can be described as \[ \mathrm{Im}(\rho )=(H_1\ast \cdots \ast H_k)\ast (H_{k+1}\ast \cdots \ast H_l)\ast (H_{l+1} \ast \cdots \ast H_m)\] by Kurosh's subgroup theorem. 
Then there exists a mapping class $f\in \mathcal{M}_g$ such that 
\begin{eqnarray*}
(\rho \circ f_{\ast })(\alpha _i)=e\;\; (1\leq i\leq g), \;\; 
(\rho \circ f_{\ast })(\beta _i)=
\begin{cases}
1\in H_i\cong \Z \;\;\; (1\leq i\leq k), \\
1\in H_i\cong \Z_3 \;\; (k+1\leq i\leq l), \\
1\in H_i\cong \Z_2 \;\; (l+1\leq i\leq m), \\
e \;\;\; (m+1\leq i\leq g),   
\end{cases}
\end{eqnarray*}
where $e$ denotes the identity element of $PSL(2; \Z)$. 
In particular, for any two homomorphisms $\rho_1, \rho_2 \colon \pi _1(\Sigma _g)\to PSL(2; \Z)$, 
the following two conditions are equivalent: 
\begin{enumerate}
\item
$\mathrm{Im}(\rho _1)=\mathrm{Im}(\rho _2)$.
\item
There exists a mapping class $f\in \mathcal{M}_g$ such that $\rho_2=\rho_1\circ f_{\ast }$.  
\end{enumerate}
\end{theorem}

We call $\rho\circ f_{\ast }$ in Theorem~\ref{hom-psl} the normal form of $\rho $ with respect to the action of the mapping class group, or simply, we say that it is of the normal form. 
Now we define the map $$p_{\ast }\colon \mathrm{Hom}(\pi _1(\Sigma _g), SL(2; \Z))
\to \mathrm{Hom}(\pi _1(\Sigma _g), PSL(2; \Z))$$ by $p _{\ast }(\widetilde{\rho })=p \circ \widetilde{\rho }$. 
Then for each homomorphism $\rho\colon \pi _1(\Sigma _g)\to PSL(2; \Z)$, 
$p _{\ast }^{-1}(\rho )$ is the set of all lifts $\widetilde{\rho}$ of $\rho $ with respect to $p$, 
and we have $\# p _{\ast }^{-1}(\rho )=2^{2g}$. 
When $\rho $ is of the normal form, we obtain the following result about the number of orbits of the action of $\mathcal{M}_g$ to all the $2^{2g}$ lifts of $\rho $. 

\vspace{8pt}

\begin{theorem}~\label{hom-sl}
Let $\rho \in \mathrm{Hom}(\pi _1(\Sigma _g), PSL(2;\Z))$ be of the normal form. 
If $m>l$, then there exist just $2^k$ $\mathcal{M}_g$-orbits in $p_{\ast }^{-1}(\rho )$ 
corresponding to the choices of the signs of 
$$\widetilde{\rho}(\alpha _1)=\pm E_2, \; \cdots , \; \widetilde{\rho}(\alpha _k)=\pm E_2.$$
On the other hand, if $m=l$, then there exist $2^{k+1}$ orbits in $p_{\ast }^{-1}(\rho )$, 
whose details are as follows;  
there are $2^k$ orbits with $-E_2\not \in \Im (\widetilde{\rho})$ corresponding to the choices of 
$\widetilde{\rho}(\beta _1), \; \cdots , \; \widetilde{\rho}(\beta _k)$, 
and another $2^k$ orbits with $-E_2\in \Im (\widetilde{\rho})$ corresponding to the choices of the signs of
$$\widetilde{\rho}(\alpha _1)=\pm E_2, \; \cdots , \; \widetilde{\rho}(\alpha _k)=\pm E_2. $$
\end{theorem}

The facts obtained by applying the results so far to the monodromies of orientable $T^2$-bundles over $\Sigma _g$ are summarized as follows. 
First, for any orientable $T^2$-bundle $\xi $ over $\Sigma _g$, 
the monodromy $$\widetilde{\rho }\in \Hom (\pi _1(\Sigma _g), SL(2;\Z))$$ 
is uniquely determined up to the actions of $\mathcal{M}_g$ and $SL(2;\Z)$. 
Applying Corollary~\ref{sl} to this $\widetilde{\rho }$, we can show that $\xi $ is decomposable into the fiber connected sum of $g$ pieces of $T^2$-bundles over $T^2$ (Theorem~\ref{thm:decomp}). 
According to Theorem~\ref{hom-psl}, the decomposition as a fiber connected sum can be taken in the form that respects the free product decomposition of the subgroup $\Im (p\circ \widetilde{\rho })\leq PSL(2;\Z)$. 
Based on these results, we have Theorem~\ref{hom-sl}, 
which is the classification of $\Hom (\pi _1(\Sigma _g), SL(2;\Z))$ up to the action of $\mathcal{M}_g$. 
Thus, taking into account of the conjugate action of $SL(2;\Z)$, 
Theorem~\ref{hom-sl} gives the complete classification of monodromies of orientable $T^2$-bundles. 
Adding a simple consideration about the Euler classes to it, 
we obtain the following classification of isomorphism classes of orientable $T^2$-bundles. 
This is our main theorem. 

\vspace{8pt}

\begin{theorem}~\label{main thm}
Let $\xi _1$ and $\xi _2$ be any two orientable $T^2$-bundles over $\Sigma _g$, and 
$$\widetilde{\rho} _1, \widetilde{\rho}_2\in \mathrm{Hom}(\pi _1(\Sigma _g), SL(2;\Z))$$ 
their monodromy representations. Suppose that both $p \circ \widetilde{\rho} _1$ and $p \circ \widetilde{\rho}_2$ are of the normal forms in the sense of Theorem~\ref{hom-psl}, 
and that $\xi _1$ and $\xi _2$ are described as 
\[ M(\ve _1E_2, B_1, \ldots , \ve _gE_2, B_g; m, n), \;\; M(\delta _1E_2, C_1, \ldots , \delta _gE_2, C_g; k, l), \] 
where $k,l,m,n\in \Z$, and for each $i$, $\ve _i$ and $\delta _i$ denote either $1$ or $-1$. 
Then $\xi _1$ and $\xi _2$ are isomorphic to each other if and only if there exists $Q\in SL(2; \Z)$ satisfying the following conditions: 
\begin{enumerate}
\item
$\mathrm{Im}(\widetilde{\rho}_1)=Q \mathrm{Im}(\widetilde{\rho}_2) Q^{-1}$. 
\item
The two lifts $\widetilde{\rho} _1$ and $Q \widetilde{\rho} _2 Q^{-1}$ of $p \circ \widetilde{\rho} _1$ are in the same $\mathcal{M}_g$-orbit. 
\item 
When $\ve_1=\cdots =\ve_g=1$, there exist ${\x}_1, \ldots, {\x}_g\in \Z^2$ such that 
\[ \begin{pmatrix}m\\n\end{pmatrix}-Q\begin{pmatrix}k\\l\end{pmatrix}=\sum _{i=1}^g(B_i-E_2){\x}_i, \]
and otherwise, there exist ${\x}_0, {\x}_1, \ldots, {\x}_g\in \Z^2$ such that 
\[\begin{pmatrix}m\\n\end{pmatrix}-Q\begin{pmatrix}k\\l\end{pmatrix}
=2{\x}_0+\sum _{i=1}^g(B_i-E_2){\x}_i.  \] 
\end{enumerate} 
\end{theorem}

Now let us reconsider about the existence of compatible symplectic structures on a $T^2$-bundle over a surface 
under the circumstance that the classification of the isomorphism classes has been done. Then the condition of Geiges and Walczak that we mentioned at the beginning can be briefly summarized as follows. 

\begin{theorem}[Theorems~\ref{Euler} and~\ref{compatible}]
Let $g$ be a non-negative integer. Then an orientable $T^2$-bundle $$\pi \colon M(A_1, B_1, \ldots , A_g, B_g; m, n)\to \Sigma _g$$ admits a compatible symplectic structure if and only if its Euler class is a torsion. 
This condition is also equivalent to that $\pi $ is not isomorphic to $$M(E_2, E_2, \ldots , E_2, E_2; m, 0) \; (m\ne 0) \;\; \text{nor} \;\; M(E_2, C^k, \ldots , E_2, E_2; m, n) \; (n\ne 0), $$ where $C=\begin{pmatrix} 1&1 \\ 0&1 \end{pmatrix}$ and $k\in \Z$. 
\end{theorem}

Finally, we note that Theorem~\ref{hom-psl} can be generalized to the case where the range of homomorphisms is the free product of a finite number of finite cyclic groups $\Z_{k_1}\ast \cdots \ast \Z_{k_n}$ (Theorem~\ref{hom-free products}). Namely, $\Hom (\pi _1(\Sigma _g), \Z_{k_1}\ast \cdots \ast \Z_{k_n})$ can be classified in a similar way as in Theorem~\ref{hom-psl}. 
Since $\Z_{k_1}\ast \cdots \ast \Z_{k_n}$ can be seen as the fundamental group of a connected sum of lens spaces, we obtain the following application of it. 

\begin{theorem}[Corollary~\ref{lens}]
For any continuous map $\phi $ from a closed orientable surface $\Sigma _g$ 
to the connected sum of a finite number of lens spaces, 
there exists a diffeomorphism $f\in \Diff _+(\Sigma _g)$ such that 
$\phi \circ f$ can be continuously extended to the handlebody $V_g$. 
\end{theorem}

In such a way, our technique is useful also in a situation apart from $T^2$-bundles, 
so it is expected that the range of applications expand in the future. 

\vspace{8pt}

This paper consists of the former part (\S2, 3), the latter part (\S4, 5) and the applications (\S6). 
Most arguments are the fusion of algebraic viewpoints and topological viewpoints throughout this article, but if we had to say, the former half is the algebraic part and the latter half is the topological one. 
The specific organization is as follows. 
\begin{itemize}
\item
In \S2, as a preliminary to \S3, we review some group theory (\S2.1), the mapping class group of handlebody $V_g$ (\S2.2) and charts for $G$-monodromies (\S2.3). In \S3, we prove the classification of $\Hom (\pi _1(\Sigma _g), PSL(2;\Z))$ up to the action of $\mathcal{M}_g$ (Theorem~\ref{hom-psl}) using all the tools prepared in \S2. We also discuss a generalization of Theorem~\ref{hom-psl} and its application. 
\item
In \S4, as a preliminary to \S5, we summarize known results about $T^2$-bundles over surfaces. 
After recalling about orientable $T^2$-bundles over $S^1$ (\S4.1), 
we review the monodromy and the Euler class of orientable $T^2$-bundles (\S4.2), 
$SL(2;\Z)$-bundles (\S4.3) and fiber connected sums (\S4.4). 
In \S4.5, as an application of Corollary~\ref{sl}, we prove that any orientable $T^2$-bundles over $\Sigma _g$ is decomposable as the fiber connected sum of $g$ pieces of $T^2$-bundles over $T^2$ (Theorem~\ref{thm:decomp}). 
Based on these preparations, in \S5, 
we first classify $\Hom (\pi _1(\Sigma _g), SL(2;\Z))$ up to the action of $\mathcal{M}_g$ 
and use it to prove the classification theorem of isomorphism classes of orientable $T^2$-bundles over $\Sigma _g$ (Theorem~\ref{main thm}). 
\item
In \S6, we explain the applications to symplectic geometry (Theorems~\ref{Euler},~\ref{compatible}, 6.6 and 6.7). 
\end{itemize}

\section{Preliminaries}
\subsection{Some group theory}
We fix the notations as follows. 
\begin{itemize}
\item 
Let $G$ be a group and $H$ its subgroup. In this case, we write $H\leq G$. 
If $H$ is a normal subgroup of $G$, then we write $H\trianglelefteq G$. 
\item
We denote by $F_n$ the free group generated by $n$ elements $a_1, \ldots , a_n$, and its automorphism group by $\Aut(F_n)$. 
\item
Let $G$ be a group, and $S=\{ s_1, \ldots , s_k\}$ a finite subset of it. 
In this case, $\langle S \rangle$ denotes the subgroup of $G$ generated by $S$. 
If there is no fear of confusion with the presentation of the free group, 
we also denote it by $\langle s_1, \ldots , s_k \rangle$. 
On the other hand, we denote the smallest normal subgroup containing $S$ 
by $N(S)$ or $\langle s_1, \ldots , s_k \rangle ^{G}$. 
\end{itemize}
First we recall about the group structures of $SL(2;\Z)$ and $PSL(2;\Z)$. 
As is well-known, the special linear group 
\[ SL(2;\Z)=\left\{ \begin{pmatrix} a&c\\b&d \end{pmatrix} \mid a, b, c, d\in \Z, ad-bc=1 \right\} \] 
is presented as \[ SL(2;\Z)=\langle s, t \mid s^4, s^2t^{-3} \rangle , \]
where 
\[ s= \begin{pmatrix} 0 & 1 \\ -1 & 0 \end{pmatrix}, \;\;  t= \begin{pmatrix} 0 & 1 \\ -1 & 1 \end{pmatrix}. \]
Thus $SL(2;\Z)$ is isomorphic to the free product with amalgamation $\Z_4\ast _{\Z_2}\Z_6$. 
On the other hand, the special projective linear group $PSL(2;\Z)$ is the quotient of $SL(2;\Z)$ by its center $\{ \pm E_2 \}$. Let $p \colon SL(2;\Z)\to PSL(2;\Z)$ be the quotient map and set $a=p (s), \; b=p (t)$. 
Then the presentation of $PSL(2;\Z)$ is given by \[ PSL(2;\Z)=\langle a, b \mid a^2, b^3 \rangle . \]
Namely, $PSL(2;\Z)$ is isomorphic to the free product $\Z_2\ast \Z_3$. 

Next we review known facts about free groups and free product of groups. 

\vspace{8pt}

\begin{theorem}[Kurosh's subgroup theorem \cite{Ku56}]~\label{Kurosh}
Let $G_1$ and $G_2$ be groups. 
Then any subgroup $H$ of the free product $G=G_1\ast G_2$ can be described in the form 
\[ H=F(X)\ast (\ast _{i\in I} g_i A_ig_i^{-1})\ast (\ast _{j\in J} f_jB_jf_j^{-1}),  \]
where $F(X)$ is the free group generated by a subset $X\subset G$, 
and $g_i, f_j\in G$, $A_i\leq G_1$ and $B_j\leq G_2$ for any $i\in I$ and $j\in J$. 
\end{theorem}

\vspace{8pt}

\begin{theorem}[Grushko \cite{Gr40}, see also \cite{St65}]~\label{Stallings}
For any surjective homomorphism $\phi $ from the free group $F_n$ to a free product $H=H_1\ast H_2$, 
there exist subgroups $G_1$ and $G_2$ of $F_n$ such that $\phi (G_1)=H_1$, $\phi (G_2)=H_2$ and $G_1\ast G_2=F_n$. 
\end{theorem}

\begin{remark}
By Theorem~\ref{Kurosh}, the only decomposition of $F_n$ as a free product of two groups 
is described as $F_n\cong F_{i}\ast F_{n-i}$ for some $i$ with $0\leq i\leq n$. 
Thus $G_1$ and $G_2$ in Theorem~\ref{Stallings} are written as $G_1\cong F_i$ and $G_2\cong F_{n-i}$. 
More precisely, there exists an element $\gamma \in \Aut (F_n)$ such that 
\[ G_1=\gamma \left(\left< a_1, \ldots , a_i\right>\right), 
\; G_2=\gamma \left(\left< a_{i+1}, \ldots , a_n \right>\right). \] 
\end{remark}

\vspace{8pt}

\begin{theorem}[Nielsen \cite{Ni24}, see also \cite{AFV08}]~\label{AFV}
Let $i$ and $j$ be integers with $1\leq i\leq n$ and $1\leq j \leq n-1$. 
We define three automorphisms $\sigma _i$, $\tau _j$, $\eta $ of the free group $F_n$ as follows: 
\begin{eqnarray*}
\sigma _i\colon 
\begin{cases}
a_i\mapsto a_i^{-1}\\
a_k\mapsto a_k \;\; (k\ne i),  \;
\end{cases} 
\tau _j\colon 
\begin{cases}
a_j\mapsto a_{j+1}\\
a_{j+1}\mapsto a_j \\
a_k\mapsto a_k \;\; (k\ne j, j+1), \;
\end{cases}
\eta \colon 
\begin{cases}
a_1\mapsto a_2^{-1}a_1\\
a_2\mapsto a_2^{-1} \\
a_k\mapsto a_k \;\; (k> 2). 
\end{cases}
\end{eqnarray*}
Then $\sigma _i$ $(1\leq i\leq n)$, $\tau _{j}$ $(1\leq j\leq n-1)$ and $\eta $ 
form a generating system of $\Aut(F_n)$. 
\end{theorem}
\begin{remark}
Theorem~\ref{AFV} itself can be deduced from a result of Nielsen \cite{Ni24}. 
On the other hand, Armstrong--Forrest--Vogtmann \cite{AFV08} determined 
all the relations among $\sigma _i$, $\tau_j$ and $\eta $ to provide an explicit presentation of $\Aut(F_n)$.  
\end{remark}

Since the abelianization of $F_n$ is $\Z^n$, 
each element $\gamma $ of $\Aut(F_n)$ canonically induces an element $\bar{\gamma }$ of $GL(n; \Z)$. 
Notice that $\bar{\sigma }_i$, $\bar{\tau }_j$, $\bar{\eta }$ which are induced by 
the generators $\sigma _i$, $\tau_j$, $\eta $ obtained in the above theorem generate $GL(n;\Z)$, 
since they are nothing but the three types of elementary transformations. 
Therefore, for any element in $GL(n; \Z)$, there exists an element in $\Aut(F_n)$ that induces it. 

\vspace{8pt}

\begin{prop}~\label{key}
For any surjective homomorphism $\phi \colon F_n\to \Z_k$, there exists an element $\gamma \in \Aut (F_n)$
such that $(\phi \circ \gamma )(a_1)=1$ and $(\phi \circ \gamma )(a_i)=0 \;\; (2\leq i\leq n)$. 
\end{prop}
\begin{proof}
We discuss only the case where the range is $\Z$ ($k=0$). 
The $\Z$-linear map $\bar{\phi }\colon \Z^n\to \Z$ induced by $\phi $ 
is represented by a matrix \[ \begin{pmatrix} p_1&p_2& \cdots &p_n \end{pmatrix}, \] 
where $p_i$ ($1\leq i\leq n$) are integers such that their greatest common divisor is $1$. 
It can be transformed to \[ \begin{pmatrix} 1&0& \cdots &0 \end{pmatrix} \]
by a finite sequence of elementary column operations. 
By composing those operations, we obtain an invertible matrix $\bar{\gamma }\in GL(n; \Z)$. 
Then an element  $\gamma \in \Aut (F_n)$ that induces $\bar{\gamma }$ satisfies desired conditions. 
The same argument works even when $k\ne 0$. 
\end{proof}

\subsection{Mapping class group of handlebody of genus $g$}~\label{MGC}

Let $\mathcal{H}_g$ denote the mapping class group of the handlebody $V_g$. 
Since a self-diffeomorphism of $V_g$ induces that of $\Sigma _g$ when restricted to the boundary $\partial V_g=\Sigma _g$, 
$\mathcal{H}_g$ is naturally embedded into $\mathcal{M}_g$ because $V_g$ is an irreducible $3$-manifold. 
We denote the subgroup of $\mathcal{M}_g$ obtained as the image by $\mathcal{H}_g^{\ast }$. 

In this paper, the product of two elements $[f_1], [f_2]\in \mathcal{M}_g$ is defined by $$[f_1][f_2]=[f_2\circ f_1].$$
Accordingly, we regard the action of $\mathcal{M}_g$ on $\Sigma_g$ as a right action. 
This convention is necessary in order to make each monodromy a homomorphism (see \S~\ref{torus bdles}). 
In line with this, we also define the product on $\mathcal{H}_g$ and $\Aut(F_g)$ in the reverse order of composition, 
and regard their actions on $V_g$ and $F_g$ as right actions. Furthermore, $\mathcal{M}_g$ and $\mathcal{H}_g$ also act from the right on $\pi _1(\Sigma _g)$ and $\pi _1(V_g)$, respectively. Note that $\pi_1(V_g)$ is isomorphic to the free group $F_g$, since $V_g$ is homotopy equivalent to the bouquet of $g$ circles. 

Let $\iota \colon \Sigma_g \to V_g$ be the inclusion and $\iota_\ast \colon \pi_1(\Sigma_g) \to \pi_1(V_g) \cong F_g$ the induced homomorphism. Then $\ker (\iota _{\ast })$ is the smallest normal subgroup of $\pi _1(\Sigma _g)$ containing $\{ \alpha _1, \ldots, \alpha _g \}$, and $\iota _{\ast }$ maps $\beta _i^{-1}$ to the element $a_i$ of $F_g$. 

\vspace{8pt}

\begin{prop}[Griffiths \cite{Gr64}]~\label{chara}
For an element $f\in \mathcal{M}_g$, the following conditions are equivalent. 
\begin{enumerate}
\item
$f$ is an element of $\mathcal{H}_g^{\ast }$. 
\item
$f_{\ast }$ preserves the smallest normal subgroup of $\pi _1(\Sigma _g)$ 
containing $\{ \alpha _1, \ldots, \alpha _g \}$. 
\end{enumerate}
\end{prop}

Now let us see several examples. 

\vspace{5pt}

\begin{example}~\label{H_g}
The following three examples all belong to $\mathcal{H}_g$. 
\begin{enumerate}
\item
We take a separating curve $\gamma _i$ as depicted in Figure~\ref{fig:half twist 1}. 
Then the half twist $k_i$ about $\gamma _i$ belongs to $\mathcal{H}_g$, 
and its action on $\pi _1(V_g)\cong F_g$ coincides with that of $\sigma_i$. 
\item
Let $\gamma _{j, j+1}$ be a separating curve depicted in Figure~\ref{fig:half twist 2}. 
Then the half twist $d_j$ about it belongs to  $\mathcal{H}_g$, 
and its action on $\pi _1(V_g)\cong F_g$ coincides with that of $\sigma _j \sigma _{j+1}\tau _j$. 
\item
Let $\delta _1$ and $\delta _2$ be the curves depicted in Figure~\ref{fig:half twist 3}. 
Then we can define the half twist $t_1$ about them, which belongs to $\mathcal{H}_g$. 
Its action on $\pi _1(V_g)\cong F_g$ coincides with that of $\tau _1\sigma _2\eta \sigma _2\tau _1$. 
\end{enumerate}
\end{example}

\begin{figure}
\begin{center}
\includegraphics[width=11cm, pagebox=cropbox,clip]{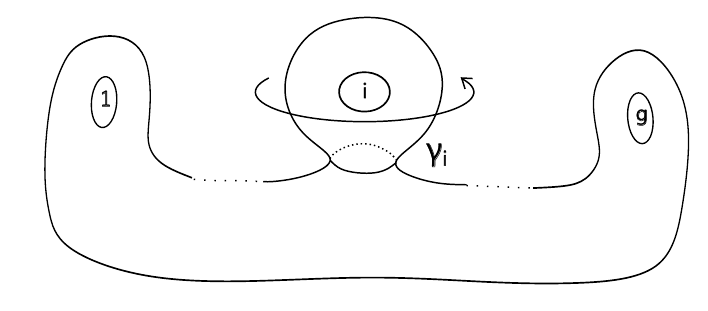}
\caption{$\gamma _i$ and the half twist $k_i$}
\label{fig:half twist 1}
\vspace{30pt}
\includegraphics[width=11cm, pagebox=cropbox,clip]{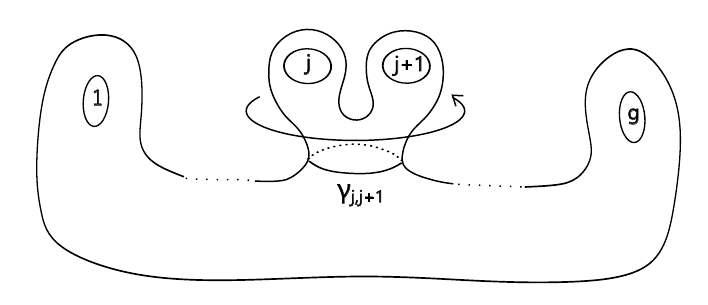}
\caption{$\gamma _{j, j+1}$ and the half twist $d_j$}
\label{fig:half twist 2}
\vspace{20pt}
\includegraphics[width=11cm, pagebox=cropbox,clip]{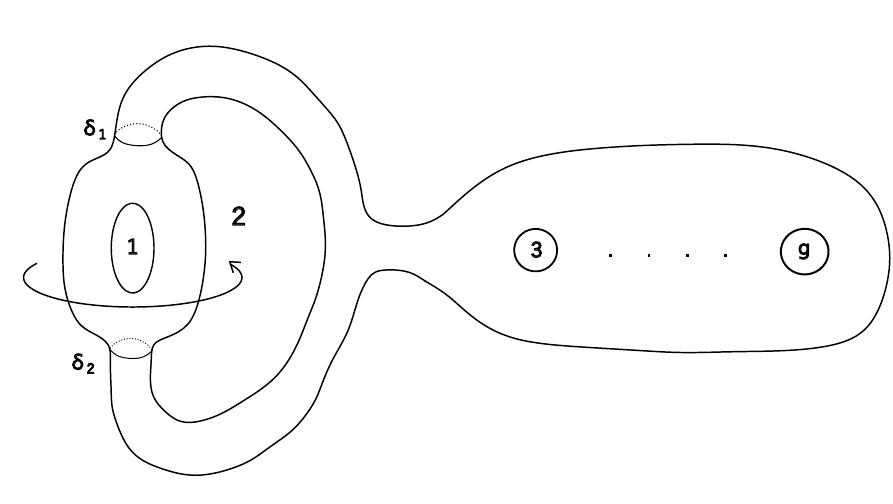}
\caption{$\delta _1$, $\delta _2$ and the half twist $t_1$}
\label{fig:half twist 3}
\end{center}
\end{figure}

The next proposition guarantees that 
the action of $\Aut (F_g)$ on $\pi _1(V_g)\cong F_g$ is recovered by that of $\mathcal{H}_g$. 

\vspace{8pt}

\begin{prop}[Griffiths \cite{Gr64}, Suzuki \cite{Su77}]~\label{F_g}
For any $\gamma \in \Aut (F_g)$, there exists a mapping class $h\in \mathcal{H}_g$ 
such that $h_{\ast }\colon \pi _1(V_g)\to \pi _1(V_g) $ coincides with $\gamma \colon F_g\to F_g$ 
under the identification $\pi _1(V_g)\cong F_g$. 
\end{prop}
\begin{proof}
When we regard the generators $\sigma _i, \tau _j, \eta $ of $\Aut(F_g)$ given in Theorem~\ref{AFV} as actions on $\pi _1(V_g)$, the elements $h_{\sigma _i}, h_{\tau _j}, h_{\eta }\in \mathcal{H}_g$ that realize them are given by 
\[h_{\sigma _i}=k_i, \;\; h_{\tau _j}=k_jk_{j+1}d_j, \;\;  h_{\eta }=k_1d_1t_1d_1k_1, \] respectively. 
\end{proof}

\subsection{G-monodromy and charts}
\begin{definition}
A group homomorphism from a surface group $\pi _1(\Sigma _g)$ to a group $G$ is called a {\it $G$-monodromy representation}. 
\end{definition}
From now on, let $G$ be a group with a finite presentation $G=\left< \mathcal{X} \mid \mathcal{R} \right>$. 
Moreover, let $\Gamma $ be an oriented graph on $\Sigma _g$ such that each edge is labelled by an element of $\mathcal{X}$. 

\begin{definition}[Intersection word]
Let $\eta \colon [0,1]\to \Sigma_g$ be a path transverse to the edges of $\Gamma$. 
The path $\eta $ transversely intersects with $\Gamma $ at a finite number of points, 
say $b_1, b_2, \ldots , b_n$, in this order when the parameter goes from $0$ to $1$. 
Let $x_i$ denote the label of the edge containing $b_i$. 
When the edge intersects with $\eta $ from left to right (resp. right to left), 
we determine the signature $\ve _i$ by $\ve _i=+1$ (resp. $\ve _i=-1$). 
Then we define the word $w_{\Gamma }(\eta )$ consisting of letters in $\mathcal{X}\cup \mathcal{X}^{-1}$ 
by $w_{\Gamma }(\eta )= x_1^{\ve_1}x_2^{\ve_2}\cdots x_n^{\ve_n}$, 
which is called the intersection word of $\eta $ with respect to $\Gamma$. 
\end{definition}

For example, in Figure~\ref{fig:交叉語} below, 
the path $l$ embedded in $\Sigma _g$ transversely intersects with three edges of $\Gamma $ 
labeled as $b$, $b$ and  $a$, respectively. 
With respect to the orientation of path $l$, the first edge intersects with $l$ from left to right, 
the second one from left to right, and the third one from right to left. 
Thus the intersection word of $l$ with respect to $\Gamma $ is given by $w_{\Gamma}(l)=b^2 a^{-1}$.

\begin{figure}[h]
 \centering
 \includegraphics[width=6cm, pagebox=cropbox,clip]{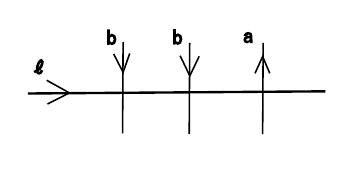}
 \caption{Intersection word of $l$}
 \label{fig:交叉語}    
\end{figure}

\begin{definition}[Charts]
A finite oriented graph $\Gamma $ on $\Sigma _g$ is called a {\it chart} 
with respect to the finitely presented group $G=\left< \mathcal{X} \mid \mathcal{R} \right>$ 
if it satisfies the following conditions: 
\begin{enumerate}
\item
Each edge is labeled by an element of $\mathcal{X}$, i.e., a generator. 
\item
For each vertex $v$, the intersection word of a small simple closed curve $m_v$ rotating around $v$ counterclockwise is a cyclic permutation of an element of $\mathcal{R}\cup \mathcal{R}^{-1}$, i.e., a relator. 
\end{enumerate}
However, we  allow a loop without vertices (given an orientation and labeled by an element of $\mathcal{X}$) as an edge of $\Gamma $. Such an edge is called a {\it hoop}. 
\end{definition}
\begin{remark}
An edge of graph is usually defined as a connection between two vertices, 
so a chart with hoops is not a graph in the normal sense. 
\end{remark}

The graph depicted in Figure~\ref{fig:Gamma} is an example of a chart on the 2-torus $T^2$ 
with respect to the group $G=\langle a,b\mid a^2,b^3\rangle$, 
where the square shown in the figure is regarded as $T^2$ by identifying its opposite edges. 
Taking a vertex $v$ as in the figure, 
then the intersection word $w_{\Gamma}(m_v)$ of a small loop $m_v$ around it is $b^{-3}$, 
which is indeed (a cyclic permutation of) an element of $\mathcal{R}\cup \mathcal{R}^{-1}$. 
\begin{figure}[h]
    \centering
    \includegraphics[width=7cm, pagebox=cropbox,clip]{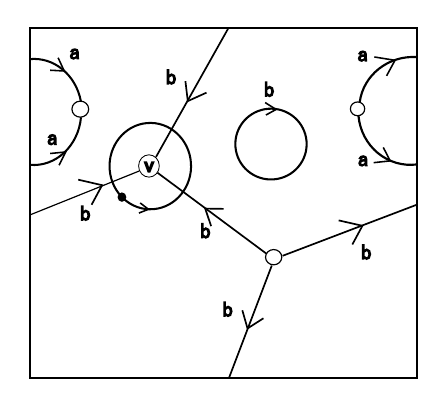}
    \caption{Example of $\Gamma$ and $w_{\Gamma}(m_{v })$}
    \label{fig:Gamma}
\end{figure}

For a given chart $\Gamma $ on $\Sigma _g$, we can define a $G$-monodromy representation 
$$\rho_{\Gamma }\colon \pi _1(\Sigma _g, b_0)\to G $$
by corresponding a loop $l$ with $l(0)=l(1)=b_0$ to its intersection word $w_{\Gamma }(l)$. 
Here we need to slightly perturb the loop $l$ so that it becomes transverse to $\Gamma $, if necessary. 
Now we explain why $\rho_{\Gamma }$ is well-defined. 
Let $l_0$ and $l_1$ be two loops such that $l_0\simeq l_1$, and $l_t$ a homotopy connecting them. 
If $l_t$ is transverse to $\Gamma $ for each $t\in [0, 1]$, then we have $w_{\Gamma}(l_0)=w_{\Gamma}(l_1)$,  since $l_t$ intersects with $\Gamma $ in topologically the same way for all $t\in [0,1]$. 
The way of intersection essentially changes only when the homotopy $l_t$ passes through a vertex $v$. 
In this case, the difference between $w_{\Gamma}(l_0)$ and $w_{\Gamma}(l_1)$ is just $w_{\Gamma}(m_v)$, 
which is a relator of $G=\langle \mathcal{X} \mid \mathcal{R} \rangle $. 
Therefore, $w_{\Gamma}(l_0)$ equals to $w_{\Gamma}(l_1)$ as an element of $G$, and thus, $\rho _{\Gamma}$ is well-defined. (It is clear that $\rho _{\Gamma}$ is homogeneous.)

\begin{example}
Let us see some examples of a chart $\Gamma $ on $T^2$ 
and the corresponding homomorphism $\rho _{\Gamma }$. 
\begin{enumerate}
\item
First we consider the case where $G=F_2=\left< a, b \right>$. 
Since there is no relation between the two generators of the free group $F_2$, 
every $F_2$-chart has no vertex, but only hoops. 
The chart on $T^2$ shown in Figure~\ref{fig:F_2} is such an example. 
The intersection words of the loops $\alpha $ and $\beta $ depicted in the figure are given by 
$$\rho_{\Gamma}(\alpha)=e, \;\; \rho_{\Gamma}(\beta)=b^{-1}a, $$ 
which determine the homomorphism $\rho _{\Gamma}\colon \pi _1(T^2)\to F_2$. 
\begin{figure}
    \centering
    \includegraphics[width=7cm, pagebox=cropbox,clip]{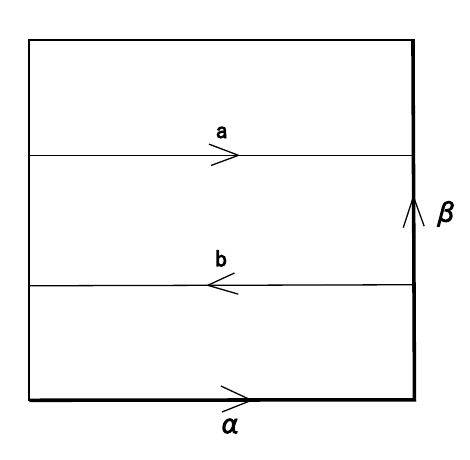}
    \caption{Example of a $G$-chart on $T^2$ (when $G=F_2=\langle a,b\rangle$)}
    \label{fig:F_2}
\end{figure}
\item
Figure~\ref{fig:enter-label} describes the same $G$-chart as that shown in Figure~\ref{fig:Gamma}, 
where $G=\Z_2\ast \Z_3=\langle a,b\mid a^2,b^3\rangle$. 
The intersection words of the loops $\alpha $ and $\beta $ are 
$$\rho_{\Gamma}(\alpha )=b, \;\; \rho_{\Gamma}(\beta )=ba^2=b, $$ respectively. 
From these, the homomorphism $\rho_{\Gamma}\colon \pi _1(T^2)\to \Z_2\ast \Z_3$ is determined. 
\begin{figure}
    \centering
    \includegraphics[width=7cm, pagebox=cropbox,clip]{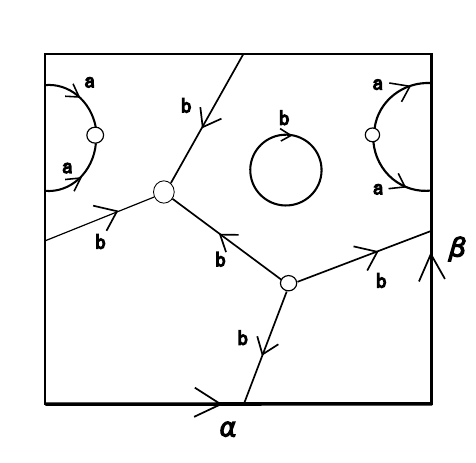}
    \caption{Example of a $G$-chart on $T^2$ ($G = \langle a,b\mid a^2,b^3\rangle$)}
    \label{fig:enter-label}
\end{figure}
\end{enumerate}
\end{example}

Thus, for a given chart $\Gamma $, the corresponding $G$-monodromy representation $\rho _{\Gamma }$ is determined. In fact, the converse is also true. 

\vspace{8pt}

\begin{theorem}[Kamada \cite{Ka07}, Hasegawa \cite{Ha06}, see also \cite{EKHT15}]~\label{chart}
For any $G$-monodromy representation $\rho $, 
there exists a $G$-chart $\Gamma $ such that $\rho _{\Gamma }=\rho $. 
\end{theorem}

\begin{proof}[{\bf Outline of the proof}]
For a given  $\rho $, we construct $\Gamma $ as follows. 
Let $b_0$ be the base point of $\Sigma_g$. 
We take a generating system $\alpha _1, \beta _1, \ldots , \alpha _g, \beta _g$ of the fundamental group 
$\pi _1(\Sigma _g, b_0)$ as depicted in Figure~\ref{fig:generators}, 
and the standard cell decomposition of $\Sigma_g$ corresponding to it. 
Namely, the cell decomposition has $b_0$ as a unique $0$-cell $e^0$, 
$1$-cells $e^1_1, e^1_2, \ldots , e^1_{2g-1}, e^1_{2g}$ corresponding to 
$\alpha _1, \beta _1, \ldots , \alpha _g, \beta _g$, respectively, 
and a unique $2$-cell $e^2$ attached along the loop $[\alpha _1, \beta _1]\cdots [\alpha _g, \beta _g]$. 
For each $i$ with $1\leq i\leq g$, let $w(\alpha _i)$ and $w(\beta _i)$ be the words of letters in 
$\mathcal{X}\cup \mathcal{X}^{-1}$ corresponding to $\rho (\alpha _i)$ and $\rho (\beta _i)$, respectively.  
Then we arrange parallel oriented arcs labelled by elements of $\mathcal{X}\cup \mathcal{X}^{-1}$ on a neighborhood of each $1$-cell so that they are transverse to the $1$-cell and the intersection word of $e^1_{2i-1}$ (resp. $e^1_{2i}$) coincides with $w (\alpha _i)$ (resp. $w (\beta _i)$). 
Since $\rho $ is a homomorphism, 
the intersection word of the loop $\partial e^2$ with respect to these arcs represents 
\[ [\rho (\alpha _1), \rho (\beta _1)] \cdots [\rho (\alpha _g), \rho (\beta _g)] 
=\rho \big([\alpha _1, \beta _1]\cdots [\alpha _g, \beta _g]\big)=\rho (e)=e \] in $G$. 
Hence, this word can be transformed into an empty word by a finite iteration of the following operations;  
\begin{enumerate}
\item 
deletion or insertion of trivial relation $x^{\ve}x^{-\ve}\; (x\in{\mathcal{X}},\ve \in{\{\pm 1\}})$, 
\item 
insertion of a relation $r^{\ve}\; (r\in{\mathcal{R},\ve \in{\{\pm 1\}}})$. 
\end{enumerate}
The operation (1) corresponds either connecting two adjacent oriented arcs labelled by the same letter $x$ or inserting a new oriented arc labelled by $x$. 
The operation (2) corresponds to the insertion of a vertex representing the relation $r^{\ve}$. 
Therefore, we can extend the arcs arranged on a neighborhood of the $1$-skeleton according to the algebraic operations that transform the intersection word of $\partial e^2$ into an empty word. 
Then there exists a $2$-disk $D$ inside the $2$-cell $e^2$ whose boundary does not intersect with extended arcs. This implies that we have already obtained a graph on $\Sigma _g\setminus D$. 
Now gluing the $2$-disk $D$ without any vertex or edge, 
then we obtain a chart $\Gamma $ on $\Sigma _g$. 
It is clear by construction that $\rho_{\Gamma }=\rho$. 
\end{proof}

\vspace{5pt}

\begin{remark}~\label{rem:noncontractible}
In fact, operations of a chart that do not change its monodromy, 
which are called chart moves of type W, 
have been completely classified in \cite{Ka07}. 
Moreover, it has been proven in the same paper that a chart $\Gamma $ satisfying $\rho_{\Gamma }=\rho $ uniquely exists up to those operations. 
In particular, deletion of a contractible loop from a chart $\Gamma $ does not change the monodromy. 
In what follows, we always assume that a chart $\Gamma $ does not contain contractible loops. 
\end{remark}

\begin{remark}
Changing the base point $b_0$ affects the $G$-monodromy by conjugate action of $G$. 
When $G=\Diff (F)$ and $\rho \colon \pi _1(\Sigma _g)\to \Diff (F)$ is the monodromy of an $F$-bundle, 
such an ambiguity can be recovered by changing the identification between the fiber over the base point and $F$. Hence, in the following, we always omit the base point $b_0$. 
\end{remark}

When $G=PSL(2;\Z)\cong \Z_2\ast \Z_3=\langle a,b\mid a^2,b^3\rangle$, 
any vertex of a $G$-chart $\Gamma $ is either degree-$2$ or degree-$3$ 
according as it represents $a^2$ or $b^3$. 
Since a degree-$2$ vertex and a degree-$3$ vertex are never connected by an edge, 
the connected component of a degree-$2$ vertex always forms a simple closed curve on $\Sigma _g$. 
This fact plays an important role in the proof of Theorem~\ref{psl}.

\section{Classification of $\mathrm{Hom}(\pi _1(\Sigma _g), PSL(2; \Z))$}
In this section, we first prove Theorem~\ref{psl} by an ingenious usage of charts. 
Then we prove Theorem~\ref{hom-psl} and its generalization (Theorem~\ref{hom-free products}) 
by combining several known results in group theory and 
consideration about the mapping class group $\mathcal{H}_g$ of the handlebody $V_g$.

\begin{proof}[{\bf Proof of }$\mathrm{{\bf Theorem~\ref{psl}}}$]
We prove it by induction on $g$. 
For any homomorphism 
\[ \rho\colon \pi _1(\Sigma _g)\to PSL(2; \Z)\cong \Z_2\ast \Z_3, \] 
we take a chart $\Gamma $ of $\rho $ 
with respect to the finite presentation $PSL(2; \Z)=\left< a, b \mid a^2, b^3 \right>$. 
Let $X$ be the set of simple closed curves on $\Sigma _g$ 
consisiting of  the connected components of degree-$2$ vertices and hoops. 
\\
{\bf (A)} The case $g=1$. 
\begin{enumerate}
\item
If $X=\emptyset $, then $\rho $ can be treated as a homomorphism to the abelian group $\Z_3$, 
and thus, it can be described as \[\rho(\alpha)=b^i, \;\; \rho (\beta )=b^j \quad (i,j\in{\{0,1,2\}}),\]
where $b$ is the canonical generator of $\Z_3$. 
Hence, there exist coprime integers $p$ and $q$ such that $\rho (\alpha ^p\beta ^q)=e$. 
Then we obtain integers $r$ and $s$ such that $ps-qr=1$ by the Euclidean algorithm. 
Now taking the mapping class $f\in \mathcal{M}_1$ 
represented by $\begin{pmatrix}p&r\\q&s\end{pmatrix}\in SL(2;\Z)$, 
we have $$(\rho\circ f_{\ast })(\alpha )=\rho (f_{\ast }(\alpha ))=\rho (\alpha ^p\beta ^q)=e. $$
\item
If $X\ne \emptyset $, we take an element $c\in X$. 
Then $c$ is a simple closed curve on $T^2$, which is not contractible by Remark~\ref{rem:noncontractible}. 
Hence, there exist coprime integers $p$ and $q$ such that $c=\alpha ^p\beta ^q\in \pi _1(T^2)$. 
Since $c$ is a connected component of a finite graph $\Gamma $, there is a simple closed curve $c'$ 
parallel to $c$ such that $\Gamma \cap c'=\emptyset $. 
Then the monodromy along $c'$ is trivial, and we have $\rho (\alpha ^p\beta ^q)=\rho (c')=e$. 
Now by the same argument as in (1), we obtain $r, s\in \Z$ with $\begin{pmatrix}p&r\\q&s\end{pmatrix}\in SL(2;\Z)$ and the corresponding mapping class $f\in \mathcal{M}_1$ so that we have 
$$(\rho\circ f_{\ast })(\alpha )=\rho (f_{\ast }(\alpha ))=\rho (\alpha ^p\beta ^q)=e.$$
\end{enumerate}
{\bf (B)} 
The case $g\geq 2$. 
We are going to prove that the assertion holds for $\Sigma _g$ 
under the assumption that it is true for $\Sigma _{g-1}$. 
In order for that, it is enough to show 
the existence of a separating curve $\gamma $ with $\rho (\gamma )=e$ 
that separates $\Sigma _g$ into the connected sum of $\Sigma _{g-1}$ and $T^2$. 
We will discuss the following three cases: \\
(1) $X=\emptyset $, \\
(2) $X\ne \emptyset $ and $X$ contains a non-separating curve, \\
(3) $X\ne \emptyset $ and all the members of $X$ are separating curves. 

\begin{figure}
\begin{center}
\includegraphics[width=9cm, pagebox=cropbox,clip]{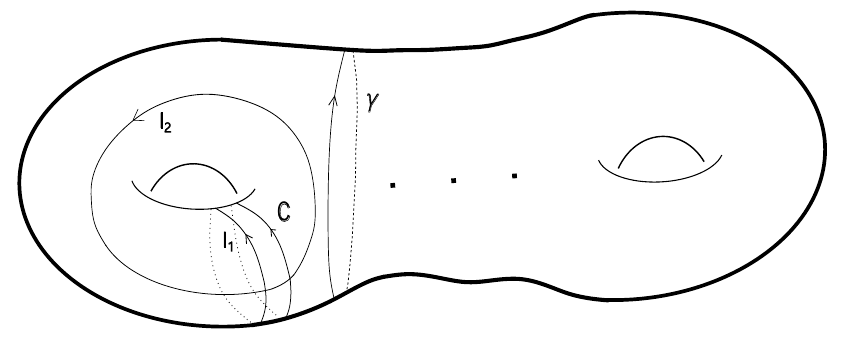}
\caption{How to take $l_1, l_2$ and $\gamma $ (when $c$ is non-separating)}
\label{fig:nonsep}
\includegraphics[width=9cm, pagebox=cropbox,clip]{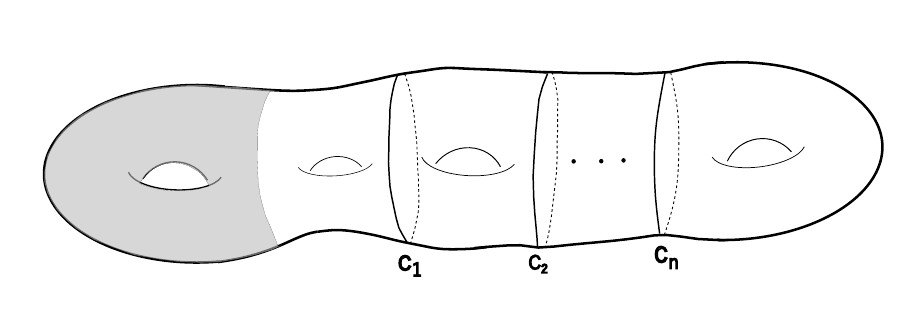}
\caption{When all the elements of $X$ are separating curves}
\label{fig:sep}
\end{center}
\end{figure}
\begin{enumerate}
\item
In this case, $\rho $ can be treated as a homomorphism into the abelian group $\Z_3$. 
On the other hand, for any separating curve $\gamma $ that separates $\Sigma _g$ into the connected sum of $\Sigma _{g-1}$ and $T^2$, there exist $l_1 , l_2 \in \pi _1(\Sigma _g)$ such that 
\[ \gamma =[l_1 , l_2 ]=l_1l_2 l_1 ^{-1}l_2^{-1}. \]
Then we have the following by the commutativity of $\Z_3$: 
\[ \rho(\gamma )=\rho(l_1 )\rho(l_2 )\rho(l_1)^{-1}\rho(l_2)^{-1}=e. \]
\item
Let $c\in X$ be a non-separating curve. 
Since $c$ can be mapped by an element of $\mathrm{Diff}_+(\Sigma_g)$ to the location 
as depicted in Figure~\ref{fig:nonsep}, we may assume that it is originally in that place. 
Since $c$ is a connected component of a finite graph $\Gamma $, 
we can take a simple closed curve $l_1$ parallel to $c$ and satisfying $\Gamma \cap l_1 =\emptyset $. 
Then it follows that $\rho(l_1)=e$. 
Moreover, taking $l_2$ and $\gamma $ as in Figure~\ref{fig:nonsep}, we have 
\[ \rho(\gamma )=\rho ([l_2, l_1^{-1}])=[\rho (l_2), \rho(l_1)^{-1}]=[\rho (l_2), e]=e. \]
\item
When all the members $c_1, \ldots, c_n$ of $X$ are separating curves, 
there are no degree-$2$ vertices nor hoops in the shaded region in Figure~\ref{fig:sep}. 
Then taking $\gamma $ as the boundary curve of that region, 
we have $\rho (\gamma )=e$ by the same argument as in (1). 
\end{enumerate}
Thus we obtain a desired separating curve $\gamma $ in any case. 
If we separates $\Sigma _g$ into $T^2$ and $\Sigma _{g-1}$ along this curve, 
the homomorphism $\rho $ can be decomposed into the two homomorphisms 
$\rho'\colon \pi _1(T^2)\to PSL(2; \Z)$ and  $\rho'' \colon \pi _1(\Sigma _{g-1})\to PSL(2; \Z)$. 
Therefore, the assertion of this theorem holds for $\Sigma _g$ if it does for $\Sigma _{g-1}$. 
\end{proof}

\begin{proof}[{\bf Proof of} $\mathrm{{\bf Theorem~\ref{hom-psl}}}$]
Let $\rho\colon \pi _1(\Sigma _g)\to PSL(2; \Z)$ be any homomorphism. 
By Theorem~\ref{psl}, there exists a diffeomorphism $\phi \in \Diff_+(\Sigma _g)$ 
such that $(\rho \circ \phi _{\ast })(\alpha _i)=e$ for each $i$. 
Then the homomorphism $\rho'=\rho\circ \phi _{\ast }$ factorizes so that the following diagram commutes; 
\begin{equation}
\xymatrix{
&\pi _1(V_g)\ar[dr]^-{\eta }\ar@{}[d]|{}&\\
\pi _1(\Sigma _g) \ar[ru]^-{\iota_{\ast }}\ar[rr]_-{\rho '}&&PSL(2;\Z)
}
\end{equation}
that is, there exists a homomorphism 
\[\eta \colon F_g\cong \pi _1(V_g)\to PSL(2; \Z)\cong \Z_2\ast \Z_3 \]
such that $\rho'=\eta \circ \iota_{\ast }$, where $\iota \colon \Sigma _g\to V_g$ is the inclusion map. 
Now applying Kurosh's subgroup theorem to the subgroup 
$\eta (F_g)=\rho '(\pi _1(\Sigma _g))\subset \Z_2\ast \Z_3, $
it turns out that $\eta (F_g)$ is isomorphic to the free product of some copies of $\Z_2$, $\Z_3$ and $\Z$. 
Namely, there exist subgroups $H_j\subset \eta (F_g)$ such that 
\[H_1\cong \cdots \cong H_k\cong \Z, \;\; 
H_{k+1}\cong \cdots \cong H_l\cong \Z_3, \;\;  
H_{l+1}\cong \cdots \cong H_m\cong \Z_2 \] 
and 
\[ \eta (F_g)=(H_1\ast \cdots \ast H_k)
\ast (H_{k+1}\ast \cdots \ast H_l)
\ast (H_{l+1} \ast \cdots \ast H_m).\]
Then we apply Grushko's theorem to the surjective homomorphism $\eta \colon F_g\to \eta (F_g)$ $(m-1)$ times, and obtain subgroups $G_j$ of $F_g$ $(1\leq j\leq m)$ such that 
\[ \eta (G_j)=H_j \;\; \text{and} \;\; G_1\ast \cdots \ast G_m=F_g. \] 
Since $G_j$ is a finitely generated free group and $H_j$ is isomorphic to either $\Z$, $\Z_3$ or $\Z_2$, 
we obtain $\gamma \in \Aut (F_g)$ such that 
\begin{eqnarray*}
(\eta \circ \gamma )(a _i)=
\begin{cases}
1\in H_i\cong \Z \;\;\; (1\leq i\leq k), \\
1\in H_i\cong \Z_3 \;\; (k+1\leq i\leq l), \\
1\in H_i\cong \Z_2 \;\; (l+1\leq i\leq m), \\
e \;\;\; (m+1\leq i\leq g) 
\end{cases}
\end{eqnarray*}
by applying Proposition~\ref{key} repeatedly. 
Moreover, there exists $h\in \mathrm{Diff}_+(V_g)$ satisfying $h_{\ast }=\gamma $ by Proposition~\ref{F_g}. 
Then we have 
\begin{eqnarray*}
(\eta \circ \gamma )\circ \iota_{\ast }
=\eta \circ h_{\ast }\circ \iota_{\ast } 
=\eta \circ (h \circ \iota)_{\ast }
=\eta \circ (\iota \circ h|_{\Sigma _g})_{\ast } \\
=\eta \circ \iota_{\ast }\circ (h|_{\Sigma_g})_{\ast }
=\rho'\circ (h|_{\Sigma_g})_{\ast } =\rho \circ (\phi \circ h|_{\Sigma _g})_{\ast }.
\end{eqnarray*}
Since $[h|_{\Sigma _g}]\in \mathcal{H}_g^{\ast }$, we have  
$(h|_{\Sigma _g})_{\ast }(\alpha _i)\in \left<\alpha _1, \ldots , \alpha _g\right>^{\pi_1(\Sigma _g)}$
by Proposition~\ref{chara}. Then it follows that 
\[ \left(\rho \circ (\phi \circ h|_{\Sigma _g})_{\ast }\right)(\alpha _i)\in 
\Big<(\rho\circ \phi_{\ast })(\alpha _1), \ldots , (\rho\circ \phi_{\ast })(\alpha _g)\Big>^{PSL(2;\Z)}= \{e\},  \]
that is, $\left(\rho \circ (\phi \circ h|_{\Sigma _g})_{\ast }\right)(\alpha _i)=e$. 
Now denote the mapping class $[\phi \circ h|_{\Sigma _g}]\in \mathcal{M}_g$ by $f$, 
then the assertion of the theorem follows. 
\end{proof}

Theorem~\ref{hom-psl} can be generalized into the case 
where the range of homomorphisms is the free product of finite cyclic groups 
$\Z_{k_1}\ast \Z_{k_2} \ast \cdots \ast \Z_{k_n}$. 
The following arguments (Theorem~\ref{hom-free products} and Corollary~\ref{lens}) 
are due to Masayuki Asaoka's advice. 

By Kurosh's subgroup theorem, any subgroup $H$ of $\Z_{k_1}\ast \Z_{k_2} \ast \cdots \ast \Z_{k_n}$ 
can be described as \[ H=H_1\ast H_2\ast \cdots \ast H_g. \]
Here each $H_j$ is a subgroup of $H$ with $H_i\cong \Z_{\chi (i)}$, where 
\[\chi \colon \{ 1, 2, \ldots , g \}\to 
\left(\bigcup _{j=1}^n\{\text{positive divisors of $k_j$} \}\right)\cup \{0 \}\] 
is a map and $\Z_0=\Z, \; \Z_1=\{e\}$. 

\vspace{8pt}

\begin{theorem}~\label{hom-free products}
Let $k_1, k_2, \ldots, k_n$ be integers greater than $1$. 
For any homomorphism \[ \rho\colon \pi _1(\Sigma _g)\to \Z_{k_1}\ast \Z_{k_2} \ast \cdots \ast \Z_{k_n},\]
the subgroup $\mathrm{Im}(\rho )\subset \Z_{k_1}\ast \Z_{k_2} \ast \cdots \ast \Z_{k_n}$ 
can be described as 
\[ \mathrm{Im}(\rho )=H_1\ast H_2\ast \cdots \ast H_g \]
by the same argument above. 
Then there exists a mapping class $f\in \mathcal{M}_g$ such that 
\begin{eqnarray*}
(\rho \circ f_{\ast })(\alpha _i)=e\;\; (1\leq i\leq g), \;\; 
(\rho \circ f_{\ast })(\beta _i)=
\begin{cases}
1\in H_i\cong \Z_{\chi (i)} \;\; (\chi (i)\ne 1), \\
e \;\;\;\;\;\;\; (\chi(i)=1). 
\end{cases}
\end{eqnarray*}
In particular, for any two homomorphisms 
$\rho_1, \rho_2 \colon \pi _1(\Sigma _g)\to \Z_{k_1}\ast \Z_{k_2} \ast \cdots \ast \Z_{k_n}$, 
the following two conditions are equivalent. 
\begin{enumerate}
\item
$\mathrm{Im}(\rho _1)=\mathrm{Im}(\rho _2)$. 
\item
There exists a mapping class $f\in \mathcal{M}_g$ such that $\rho_2=\rho_1\circ f_{\ast }$. 
\end{enumerate}
\end{theorem}
\begin{proof}[{\bf Outline of the proof}]
First we prove the following claim, which is a generalization of Theorem~\ref{psl}. 

\vspace{5pt}

\begin{claim}~\label{3.2}
For any homomorphism $\rho\colon \pi _1(\Sigma _g)\to \Z_{k_1}\ast \Z_{k_2} \ast \cdots \ast \Z_{k_n}$, there exists a mapping class $f\in \mathcal{M}_g$ such that $(\rho \circ f_{\ast })(\alpha _i)=e$ for each $i$ with $1\leq i\leq g$. 
\end{claim}
Since $\Z_{k_1}\ast \Z_{k_2} \ast \cdots \ast \Z_{k_n}
=\langle a_1, a_2, \ldots ,a_n \mid a_1^{k_1},  a_2^{k_1}, \cdots , a_n^{k_n}\rangle$, 
a degree-$k_i$ vertex and a degree-$k_j$ vertex of a chart $\Gamma $ of $\rho $, 
corresponding to the relators $a_i^{k_i}$ and $a_j^{k_j}$ respectively, 
are never connected by an edge if $i\ne j$. 
Thus, for any $i$ with $1\leq i\leq g$, each connected component of a degree-$k_i$ vertex is isolated. 
Now we take the ribbon graph obtained by thickening $\Gamma $ and denote the set of its boundary components by $X$. 
Then a parallel argument to the proof of Theorem~\ref{psl} works. 
However, we cannot exclude the case where $X$ contains contractible loops, 
since elements of $X$ are not connected components of $\Gamma $ this time. 
Thus we have to discuss the following three cases. 
\begin{enumerate}
\item
All the elements of $X$ are contractible. 
\item
$X$ contains a non-contractible separating curve. 
\item
$X$ contains a non-separating curve. 
\end{enumerate}

In the case (1), it is easily proven that $\Gamma $ is connected. 
Then $\rho $ can be seen as a homomorphism to a single abelian group $\Z_{k_i}$, 
and thus, the same argument as in the proof of Theorem~\ref{psl} works. 
In the case (2), we take a non-contractible separating curve $c\in X$. 
Then it follows that $\Gamma \cap c=\emptyset$, so we separates $\Sigma _g$ along $c$. 
Here we have to notice that $\Sigma _g$ is not necessarily decomposed into $T^2$ and $\Sigma _{g-1}$, 
and the assumption of induction should be modified in the form that 
the assertion holds for all the closed orientable surfaces of genus smaller than $g$. 
In the case (3), the same argument as in the proof of Theorem~\ref{psl} (B) (2) works. 
Thus Claim~\ref{3.2} is proven. 
A parallel argument to the proof of Theorem~\ref{hom-psl} works 
when we replace $\Z_2\ast \Z_3$ by $\Z_{k_1}\ast \Z_{k_2} \ast \cdots \ast \Z_{k_n}$, 
this completes the proof of Theorem~\ref{hom-free products}. 
\end{proof}

We obtain the following application as a corollary to this theorem. 

\vspace{8pt}

\begin{corollary}~\label{lens}
For any continuous map $\phi $ from a closed orientable surface $\Sigma _g$ 
to the connected sum of a finite number of lens spaces, 
there exists a diffeomorphism $f\in \Diff _+(\Sigma _g)$ such that 
$\phi \circ f$ can be continuously extended to the handlebody $V_g$. 
\end{corollary}
\begin{proof}
The assertion follows from Theorem~\ref{hom-free products} (in fact, Claim~\ref{3.2} is enough) and 
the fact that the connected sum of lens spaces $L(p_1, q_1), L(p_2, q_2),  \cdots , L(p_n, q_n)$ satisfies 
\begin{eqnarray*}
\pi _1\Big( L(p_1, q_1)\# L(p_2, q_2)\# \cdots \#L(p_n, q_n) \Big) 
&\cong &\Z_{p_1}\ast \Z_{p_2}\ast \cdots \ast \Z_{p_n}, \\
\pi _2\Big( L(p_1, q_1)\# L(p_2, q_2)\# \cdots \#L(p_n, q_n) \Big) &=&0. 
\end{eqnarray*}
\end{proof}

\section{Orientable $T^2$-bundles}~\label{torus bdles}
In this section, we first review basic facts and known results about orientable $T^2$-bundles over $S^1$ in \S~4.1, and about orientable $T^2$-bundles over $\Sigma _g$ in \S4.2, 4.3, and 4.4. 
Moreover, in \S~4.5, we show a new result obtained as an application of Corollary~\ref{sl}. 

Let $B$ be a connected orientable $C^{\infty }$-manifold 
and $\pi \colon E\to B$ an orientable $T^2$-bundle over $B$. 
Then its structure group $\Diff_+(T^2)$ is known to be homotopy equivalent 
to $\mathrm{Aff}_+(T^2)=T^2\rtimes SL(2;\Z)$ (\cite{Ha65}). 
In particular, we have 
$$\pi_0 \left(\Diff_+(T^2)\right)\cong SL(2;\Z), \; \pi_1 \left(\Diff_+(T^2)\right)\cong \Z^2,$$
and thus, the mapping class group of $T^2$ is $\mathcal{M}_1\cong SL(2;\Z)$. 

Now we take a base point $b_0$ on the base space $B$ of the bundle $\pi $ 
and fix a diffeomorphism $$\Psi \colon T^2\to F_0:=\pi^{-1}(b_0).$$ 
Let $l\colon [0,1]\to B$ be a loop with $l(0)=l(1)=b_0$. 
Since the pullback of the bundle $\pi \colon E\to B$ by $l$ is a $T^2$-bundle over $[0,1]$, which is trivial, 
there exists a bundle map $\phi \colon [0,1]\times T^2\to E$ that covers $l\colon [0,1]\to B$ and 
satisfies $\phi_0=\Psi $, where we set $\phi (t, p)=\phi _t(p)$. 
Then the composition $\Psi^{-1}\circ\phi_1\colon T^2\to T^2$ is a diffeomorphism and its isotopy class $[\Psi^{-1}\circ\phi_1]\in \mathcal{M}_1$ depends only on 
the isotopy class of $\Psi $ and the homotopy class $[l]$ of $l$. 
The mapping class $[\Psi^{-1}\circ\phi_1]$ is called the monodromy along $[l]$ with respect to $\Psi$. 
The map $$\rho\colon \pi_1(B)\to  \mathcal{M}_1\cong SL(2;\Z)$$ 
defined by sending $[l]$ to $[\Psi^{-1}\circ\phi_1]$ is called 
the monodromy of $\pi \colon E\to B$ with respect to $\Psi $. 
Since we have chosen to regard the action of $\mathcal{M}_1$ on $T^2$ as a right action (see \S~\ref{MGC}), 
$\rho$ becomes a group homomorphism.

\vspace{5pt}

\begin{remark}~\label{rem:monodromy}
The mapping class group $\mathcal{M}(B)$ acts on $\pi _1(B)$ by $\alpha \cdot [f]=f_{\ast }(\alpha )$, 
and $SL(2;\Z)$ acts on itself by $Q\cdot A=QAQ^{-1}$. 
Notice that the monodromy representation $\rho \colon \pi_1(B)\to SL(2;\Z)$ is uniquely determined 
by the isomorphism class of $\pi $ up to these actions of $\mathcal{M}(B)$ and $SL(2;\Z)$. 
Here $[f]\in \mathcal{M}(B)$ and $Q\in SL(2;\Z)$ correspond to the base map of a bundle isomorphism and the ambiguity of the isotopy class of an identification $\Psi \colon T^2\to F_0$, respectively. 
\end{remark}

\begin{remark}
Since we have decided that the product in $SL(2;\mathbb{Z})$ is defined in the reverse order of composition and 
the action of $SL(2;\mathbb{Z})$ on $T^2$ is a right action, one would, strictly speaking, have to rewrite all the notation concerning products of matrices by taking transposes.
That is, conventions such as “vectors are regarded as row vectors on which matrices act from the right” and “the product $AB$ of matrices corresponds, under the usual definition, to $\T B\T A$” would have to be adopted.
However, this would be highly confusing. Moreover, since in the subsequent discussion there will be no point at which the order of multiplication in $SL(2;\mathbb{Z})$ causes any issue, we shall henceforth employ the standard notation for products of matrices. 
\end{remark}

\subsection{Orientable $T^2$-bundles over $S^1$}
Before dealing with $T^2$-bundles over surfaces, we first consider the case $B=S^1$. 
In this case, the total space of a $T^2$-bundle is given by a mapping torus.  
Here the mapping torus $M(A)$ of $A\in SL(2;\Z)$ is defined as follows; 
$$M(A)=\R\times T^2/ \big(t, [A\x ] \big) \sim \big(t+1, [ \x ] \big). $$
The map $\pi \colon M(A)\to S^1=\R/\Z$ defined by $\pi \big(\big[t, [\x]\big]\big)=[t]$ 
is indeed a $T^2$-bundle over $S^1$ with monodromy $A$. 
Conversely, any orientable $T^2$-bundle over $S^1$ can be described in this way, 
and its isomorphism class is determined by the conjugacy class of the monodromy $A\in SL(2;\Z)$ 
(see Remark~\ref{rem:monodromy}). 

Next we review on bundle automorphisms of $M(A)$. 
The identity map $$\mathrm{id}_{M(A)}\colon M(A)\to M(A)$$ 
is clearly a bundle isomorphism covering $\mathrm{id}_{S^1}$. 
On the other hand, for any $\u \in \Z^2$, 
the map $\widetilde{f}_{\u }\colon M(A)\to M(A)$ defined by 
$$\widetilde{f}_{\u }\big(\big[t, [\x] \big]\big)=\big[t, [\x+t\u ] \big]$$ 
is also a bundle isomorphism covering $\mathrm{id}_{S^1}$. 
This is nothing but the bundle isomorphism that sends the trivial section $s_0\colon S^1\to M(A)$ defined by $s_0([t])=\big [t, [{\bf 0}]\big]$ to another one $s_{\u} \colon S^1\to M(A)$ defined by $s_{\u}([t])= \big[t, [t\u] \big]$. 

\subsection{Monodromy and Euler class (the case $B=\Sigma _g$)}
In the following, we deal with the case where $B=\Sigma _g$. 
The meaning of the monodromy of an orientable $T^2$-bundle 
$\pi \colon M^4\to \Sigma _g$ is interpreted as follows. 
First we take the standard cellular decomposition 
$$\Sigma _g=\left(e^0\cup (e^1_1\cup e^1_2\cup \cdots \cup e^1_{2g-1}\cup e^1_{2g})\right)\cup e^2$$ 
used in the proof of Theorem~\ref{chart}. 
Then let $$\{ U^0, U^1_1, U^1_2, \cdots , U^1_{2g-1}, U^1_{2g}, U^2 \}$$ 
be its associating open covering of $\Sigma _g$ that satisfies the following conditions: 
\begin{enumerate}
\item
$U^0$ is an open neighborhood of $b_0$, 
and $U^1_k$ an open neighborhood of $e^1_k\setminus U^0$ for each $k$ with $1\leq k\leq 2g$. 
\item
$U^0$, $U^2$ and each $U^1_k$ are all diffeomorphic to an open $2$-disk. 
\item
Each $U^0\cap U^1_k$ is the disjoint union of two open $2$-disks, 
and $U^1_i\cap U^1_j=\emptyset \; (i\ne j)$. 
\item
$N:=U^0\cup (U^1_1\cup \cdots \cup U^1_{2g})$ is a regular neighborhood of the $1$-skeleton. 
\item
$N\cap U^2$ is a collar neighborhood of $\partial N$. 
\end{enumerate}

By the condition (2), the $T^2$-bundle $\pi $ is trivial over $U^0$, $U^2$ and each $U^1_k$. 
We take a local trivialization $\phi _0\colon \pi^{-1}(U^0)\cong U^0\times T^2$ 
so that $\phi_0 |_{F_0}=\Psi ^{-1}$. 
Then the transition function $g_{01}^k$ over $U^0\cap U^1_k$ is defined for each $k$. 
By the condition (3), we have open sets $V_k$ and $W_k$ diffeomorphic to an open $2$-disk such that 
$$U^0\cap U^1_k=V_k\sqcup W_k, $$
If we take the transition function $g_{01}^k\colon V_k\sqcup W_k\to \Diff_+(F)$ 
that coincides with the constant map to $\mathrm{id}_{T^2}$ over $V_k$, 
then ${{g}_{01}^k}|_{W_k}$ corresponds to the monodromy along $e^1_k$. 
Since a homotopy of transition functions does not change the bundle isomorphism class, 
we may retake ${{g}_{01}^k}|_{W_k}$ as a constant map valued in $\pi _0(\Diff_+(T^2))=SL(2;\Z)$. 
Thus we obtain an element in $SL(2;\Z)$ for each $e^1_k$ with $1\leq k\leq 2g$. 
Denote it by $A_i$ if $k=2i-1$, and by $B_i$ if $k=2i$. 
Then we have 
\[ \rho (\alpha _1)=A_1, \; \rho (\beta _1)=B_1, \; \cdots, \; \rho (\alpha _g)=A_g, \; \rho (\beta _g)=B_g \]
and $[A_1, B_1]\cdots [A_g, B_g]=E_2$. In other words, the restriction of $\pi $ over $e^1_k$ is isomorphic to $M(A_i)$ if $k=2i-1$ and to $M(B_i)$ if $k=2i$. 
Since these data determine the isomorphism class of the $T^2$-bundle $\pi |_{\pi^{-1}(N)}$ over $N$, the monodromy $\rho $ can be considered as the description of nontriviality of the bundle over the $1$-skeleton. 

What is left for us is to describe how to glue 
the trivial bundle $U^2\times T^2\to U^2$ and $\pi |_{\pi^{-1}(N)}\colon \pi^{-1}(N)\to N$. 
To do so, we can restore the isomorphism class of the $T^2$-bundle $\pi $. 
Since this information of the gluing is described by the transition function $N\cap U^2\to \Diff_+(T^2)$, 
we obtain the corresponding element in $\pi _1\left(\Diff_+(T^2)\right)\cong \Z^2$, say $(m, n)$. 
Thus the isomorphism class of an orientable $T^2$-bundle $\pi \colon M^4\to \Sigma _g$ 
is represented by $A_1, B_1, \ldots , A_g, B_g\in SL(2;\Z)$ with $[A_1, B_1]\cdots [A_g, B_g]=E_2$, 
and $(m, n)\in \Z^2$. 
Hence, we may denote this $T^2$-bundle by $M(A_1, B_1, \ldots , A_g, B_g; m, n)$. 
Now the following is obvious. 

\vspace{8pt}

\begin{prop}~\label{model} 
For any orientable $T^2$-bundle $\xi $ over $\Sigma _g$, 
there exist $A_1$, $B_1$, $\ldots ,$ $A_g$, $B_g\in SL(2;\Z)$ with $[A_1, B_1]\cdots [A_g, B_g]=E_2$ and $(m, n)\in \Z^2$ such that $\xi $ is isomorphic to the bundle 
$$\pi \colon M(A_1,B_1, \ldots, A_g, B_g; m,n)\to \Sigma _g. $$
\end{prop}

Next, we explain the definition of the Euler class of an orientable $T^2$-bundle 
$$\xi =\left(M(A_1,B_1, \ldots, A_g, B_g; m,n), \pi , \Sigma _g, T^2\right).$$ 
In order for that, we consider whether the bundle 
$$\pi \colon M(A_1, B_1, \ldots , A_g, B_g; m, n)\to \Sigma _g$$ admits a cross section. 
There exists a trivial cross section over the $1$-skeleton corresponding to $s_0$. 
It extends over the $2$-cell $e^2$ if and only if $(m, n)=(0,0)$, 
since it is necessary and sufficient for a continuous map $\partial D^2\to T^2$ to be extendable over $D^2$ that it is trivial as an element of $\pi _1(T^2)\cong \Z^2$. 
In this sense, it seems possible to consider $(m, n)\in \Z^2$ 
as the obstruction to the existence of a cross section of $\pi $, but this is inaccurate. 
For, even when $(m, n)\ne (0,0)$, if we retake another cross section over the $1$-skeleton, then it might be extendable over $e^2$. 
Such an inconvenience is resolved by the local system $\{ \pi _1(T^2) \}$, that is, 
the locally constant sheaf associated with the monodromy $\rho $ whose stalk over each point $b\in \Sigma _g$ is isomorphic to $\pi _1(F_b)\cong \Z^2$. 
By obstruction theory, the obstruction to the existence of a cross section of $\xi $ 
is determined as a second cohomology class with coefficients in $\{ \pi _1(T^2) \}$ 
(see \cite{St51, MS74} for details). Denote it by $e(\xi )$. 
Then it lies in $$H^2\left(\Sigma _g; \{\pi _1(T^2) \} \right) \cong \Z^2/\sim ,$$
where the equivalence relation $\sim $ is defined by 
\begin{eqnarray*}
(k_1, k_2)\sim (l_1, l_2) 
\Longleftrightarrow 
\text{there exists $\gamma \in \pi _1(\Sigma _g)$ such that}
\begin{pmatrix} l_1 \\ l_2 \end{pmatrix}
=\rho (\gamma ) \begin{pmatrix} k_1 \\ k_2 \end{pmatrix}.  
\end{eqnarray*}
Thus the obstruction class $e(\xi )$ is not an element of $\Z^2$, 
but that of the module $\Z^2/\sim $ represented by $(m, n)\in \Z^2$. 
This is the Euler class of an orientable $T^2$-bundle $\xi $. 

\begin{definition}[the Euler class]
Let $\xi $ be an orientable $T^2$-bundle over $\Sigma _g$. 
Then the local coefficient cohomology class 
$$e(\xi )\in H^2\left(\Sigma _g; \{\pi _1(T^2) \} \right)$$ defined as 
the obstruction class to the existence of a cross section of $\xi $ is called the Euler class of $\xi $. 
\end{definition}

\vspace{5pt}

\begin{remark}~\label{rem:Euler}
As is implied by the notation, 
the Euler class $e(\xi )$ is uniquely determined by the isomorphism class of $\xi $. 
Notice that the Euler classes of two orientable $T^2$-bundles $\xi $ and $\xi'$ are comparable as cohomology classes only when its monodromies $\rho $ and $\rho'$ coincide. 
Here we mean by the coincidence of $\rho $ and $\rho'$ that there exist $f\in \Diff _+(\Sigma _g)$ and $Q\in SL(2;\Z)$ such that $\rho \circ f_{\ast }=Q\rho 'Q^{-1}$ (see Remark~\ref{rem:monodromy}). 
\end{remark}

Now we are ready to give a necessary and sufficient condition 
for two orientable $T^2$-bundles over $\Sigma _g$ to be isomorphic. 

\vspace{8pt}

\begin{theorem}~\label{isom}
Let $\xi =(E, \pi , \Sigma _g, T^2)$ and $\xi '=(E', \pi ', \Sigma _g, T^2)$
be orientable $T^2$-bundles over $\Sigma _g$, whose monodromies and Euler classes 
we denote by $\rho, e(\xi )$ and $\rho', e(\xi')$, respectively. 
Then, the bundles $\xi $ and $\xi'$ are isomorphic if and only if 
there exist $f\in \mathrm{Diff}_+(\Sigma _g)$ and $Q\in SL(2; \Z)$ 
such that $\rho\circ f_{\ast }=Q\rho'Q^{-1}$ and $e(\xi )=f^{\ast }e(\xi')$. 
\end{theorem}

\begin{proof}
The only if part is obvious from Remarks~\ref{rem:monodromy} and~\ref{rem:Euler}, so we will prove the if part. Since $f$ and $Q$ correspond to the base map of a bundle isomorphism and the ambiguity of the choice of the mapping class $[\Psi ]\in \mathcal{M}_1$ determined by an identification diffeomorphism $\Psi \colon T^2\to F_0$, respectively, we only have to deal with the case where $\rho=\rho'$ and $e(\xi )=e(\xi')$. We construct a bundle isomorphism between $\xi $ and $\xi '$ under these assumptions. 
Since $\rho$ and $\rho'$ coincide, $\xi $ and $\xi'$ can be described as 
$$\pi \colon M(A_1,B_1, \ldots , A_g, B_g; m, n)\to \Sigma _g, \;\; 
\pi ' \colon M(A_1,B_1, \ldots , A_g, B_g; k, l)\to \Sigma _g.$$
By the condition $e(\xi )=e(\xi')$, there exists an element $\gamma \in \pi _1(\Sigma_g)$ 
such that $$\begin{pmatrix} k \\ l \end{pmatrix}=\rho (\gamma ) \begin{pmatrix} m \\ n \end{pmatrix}.$$ 
Since $\rho (\gamma )$ can be described as $\rho (\gamma )=C_{1}\cdots C_{r},$
where $r$ is a positive integer and $C_{i}\in \{E_2, A_1, \ldots, A_g, B_1, \ldots , B_g\}$ for each $1\leq i\leq r$, 
we have 
$$\rho (\gamma )=E_2+\sum_{i=1}^r(C_i-E_2)C_{i+1}\cdots C_r.$$
Putting $\w _i=C_{i+1}\cdots C_r\begin{pmatrix} m \\ n \end{pmatrix}$ for each $1\leq i\leq r$, then we obtain 
$$\begin{pmatrix}k\\l\end{pmatrix}
=\begin{pmatrix}m\\n\end{pmatrix}+\sum_{i=1}^r (C_i-E_2)\w _i.$$
Hence, there exist $\u_i, \v_i\in \Z^2$ ($1\leq i \leq g$) such that 
$$\begin{pmatrix}k\\l\end{pmatrix}
=\begin{pmatrix}m\\n\end{pmatrix}+\sum_{i=1}^g \left((A_i-E_2)\u_i +(B_i-E_2)\v_i\right).$$  
Then we can construct a bundle isomorphism 
$$\widetilde{f}\colon M(A_1,B_1, \ldots , A_g, B_g; m, n) \to M(A_1,B_1, \ldots , A_g, B_g; k, l)$$
covering the identity map $\mathrm{id}_{\Sigma _g}$ by Proposition~\ref{bundle iso} proved later. 
\end{proof}

Putting $g=1$ in Theorem~\ref{isom}, we obtain the following corollary. 

\vspace{5pt}

\begin{corollary}~\label{g=1}
Let \[\xi =\left(M(A, B; m, n), \pi , T^2, T^2\right), \;\; 
\xi'=\left(M(A', B'; m', n'), \pi ', T^2, T^2\right)\]
be two orientable $T^2$-bundles over $T^2$. 
Then they are bundle isomorphic if and only if 
there exist $P=\begin{pmatrix}p&r\\q&s\end{pmatrix}, Q\in SL(2; \Z)$ and ${\x}, {\y} \in \Z^2$ 
such that $QA'Q^{-1}=A^pB^q$, $QB'Q^{-1}=A^rB^s$ and  
\[\begin{pmatrix}m\\n\end{pmatrix}-Q\begin{pmatrix}m'\\n'\end{pmatrix}
=(A-E_2){\x}+(B-E_2){\y}. \] 
\end{corollary}

\subsection{$SL(2;\Z)$-bundles}
In this subsection, we introduce a nice model for the $T^2$-bundle $M(A_1, B_1, \ldots , A_g, B_g; m, n)$ 
following the manner of Sakamoto--Fukuhara \cite{SF83}, and complete the proof of Theorem~\ref{isom}. 
First we consider the case $(m, n)=(0,0)$. 
In this case, we can reduce the structure group to $SL(2;\Z)$, since the transition function $N\cap U^2\to \Diff_+(T^2)$ can be taken as the constant map to $\mathrm{id}_{T^2}$. 
Such a $T^2$-bundle is called an $SL(2;\Z)$-bundle (in particular, it is a flat $T^2$-bundle). 
An $SL(2;\Z)$-bundle $M(A_1, B_1, \ldots , A_g, B_g; 0, 0)$ can be described as follows. 
Let $\begin{pmatrix}x\\y\end{pmatrix}$ be the standard coordinates on $\R^2$, 
and $\begin{bmatrix}x\\y \end{bmatrix}$ denote the corresponding point on $T^2=\R^2/\Z^2$. 

\begin{enumerate}
\item
The case $g=0$. Let $F=T^2$ and $M(0,0)=\C P^1\times F$. 
Then the projection $\pi \colon M(0,0)\to \C P^1$ to the first factor is the trivial $T^2$-bundle over $S^2$. In particular, it is an $SL(2;\Z)$-bundle. 
\item
The case $g=1$. Let $F=T^2$ and $A, B\in SL(2;\Z)$ satisfy $AB=BA$. 
We define \[ M(A, B; 0,0)=\R^2\times F/\sim  ,\]
where 
\[\begin{pmatrix}
\begin{pmatrix}
     x+1 \\ y 
\end{pmatrix}, 
\begin{bmatrix}
     s \\ t
\end{bmatrix}
\end{pmatrix} \sim
\begin{pmatrix}
\begin{pmatrix}
     x \\ y 
\end{pmatrix}, 
\begin{bmatrix}
        A
\begin{pmatrix}
            s \\ t
\end{pmatrix}
\end{bmatrix}
\end{pmatrix}, \]
\[\begin{pmatrix}
    \begin{pmatrix}
        x \\ y+1 
    \end{pmatrix}, 
    \begin{bmatrix}
        s \\ t
    \end{bmatrix}
\end{pmatrix} \sim
\begin{pmatrix}
    
\begin{pmatrix}
    x \\ y 
\end{pmatrix}, 
\begin{bmatrix}
    B
\begin{pmatrix}
    s \\ t
\end{pmatrix}
\end{bmatrix}. 
\end{pmatrix}\] 
Then the map $\pi \colon M(A, B; 0,0) \to \R^2/\Z^2$ defined by 
\[ 
\pi 
\begin{pmatrix}
\begin{bmatrix}
\begin{pmatrix}
    x \\ y 
\end{pmatrix}, 
\begin{bmatrix}
    s \\ t
\end{bmatrix}
\end{bmatrix}
\end{pmatrix}=
\begin{bmatrix}
    x \\ y 
\end{bmatrix} \]
is an $SL(2;\Z)$-bundle over $T^2$. 
\item
The case $g\geq 2$. Let $\mathbb{D}$ be the Poincar\'{e} disk and $z$ the complex coordinate on it. 
Since $\mathbb{D}$ is the universal cover of $\Sigma _g$, 
$\pi _1(\Sigma _g)$ acts on it from the right by the deck transformation. 
Let $F=T^2$ and $A_1, B_1, \ldots , A_g, B_g\in SL(2;\Z)$ satisfy $[A_1, B_1]\cdots [A_g, B_g]=E_2$. 
We define  
\[ M(A_1, B_1, \ldots , A_g, B_g; 0, 0)= \mathbb{D}\times F / \sim ,\]
where
\[
\left(z, 
\begin{bmatrix}
     s \\ t
\end{bmatrix}
\right) \sim
\left(z\cdot \gamma ^{-1}, 
\begin{bmatrix}
\rho (\gamma )
\begin{pmatrix}
            s \\ t
\end{pmatrix}
\end{bmatrix}
\right) \;\;\; \left(\gamma \in \pi _1(\Sigma _g) \right). \]
Then the map $\pi \colon M(A_1, B_1, \ldots , A_g, B_g; 0, 0) \to \mathbb{D}/\pi _1(\Sigma _g)$ defined by 
\[ 
\pi 
\begin{pmatrix}
\begin{bmatrix}
z, 
\begin{bmatrix}
    s \\ t
\end{bmatrix}
\end{bmatrix}
\end{pmatrix}=
[z] \]
is an $SL(2; \Z)$-bundle over $\Sigma _g$. 
\end{enumerate}

Next, we prepare a model for the case where $(m, n)\ne (0,0)$. 
In the following, we denote $M(A_1,B_1, \ldots A_g, B_g;0,0)$ by $M_0$. 
Moreover, $D$ denotes a $2$-disk of a sufficiently small radius $\ve $ 
that is centered at the origin $0=[0:1]$ of $\C \subset \C P^1$, 
at $\begin{bmatrix}
1/2 \\ 1/2
\end{bmatrix}$ on the torus $T^2=\R^2/\Z^2$, or 
at $[0]$ on $\Sigma _g=\mathbb{D}/\pi _1(\Sigma _g)$ 
according as $g=0, 1$, or $g\geq 2$. 
Then we define 
\[M(A_1,B_1, \ldots , A_g, B_g;m,n) = \left(M_0- \pi^{-1}(\mathrm{Int}{D})\right) \cup_h (D\times F).  \]
Here $h\colon \pi^{-1}(\partial D) \to \partial D\times F$ is the diffeomorphism given by 
\[
h\begin{pmatrix}
\begin{bmatrix}
r (\theta), 
\begin{bmatrix}
    s \\ t
\end{bmatrix}
\end{bmatrix}
\end{pmatrix}
= 
\begin{pmatrix}
r (\theta), 
    \begin{bmatrix}
        \begin{pmatrix}
            s \\ t
        \end{pmatrix}
        + (\theta /2\pi)
        \begin{pmatrix}
            m \\ n
        \end{pmatrix}
    \end{bmatrix}
\end{pmatrix}, \]
where $r \colon S^1\to \mathbb{D}$ is the map defined by 
\begin{eqnarray*}
r (\theta) =
\begin{cases}
[\ve e^{i\theta }: 1] \;\; (g=0), \\
\big(1/2 + \varepsilon\mathrm{cos}\theta , \; 1/2 + \varepsilon\mathrm{sin}\theta \big) \;\; (g=1),\\
\ve e^{i\theta } \;\; (g\geq 2).
\end{cases} 
\end{eqnarray*}
Now we redefine the map $\pi \colon M(A_1,B_1, \ldots, A_g, B_g; m,n)\to \Sigma _g$ by 
\begin{align*}
\pi
\begin{pmatrix}
    [z_1:z_2], 
\begin{bmatrix}
    s \\ t
\end{bmatrix}
\end{pmatrix}
&=[z_1:z_2]
\;\; (g=0), \notag\\
\pi\begin{pmatrix}
\begin{bmatrix}
\begin{pmatrix}
        x \\ y
    \end{pmatrix}, 
    \begin{bmatrix}
        s \\ t
    \end{bmatrix}
\end{bmatrix}
\end{pmatrix}
&=
\begin{bmatrix}
    x \\ y
\end{bmatrix}
\;\; (g=1), \notag\\
\pi
\begin{pmatrix}
\begin{bmatrix}
    z, 
\begin{bmatrix}
    s \\ t
\end{bmatrix}
\end{bmatrix}
\end{pmatrix}
&=[z]
\;\;\;\; (g\geq 2). 
\end{align*}
Then it is an orientable $T^2$-bundle over $\Sigma _g$. 
Thus we obtain the following conclusion:  
Any orientable $T^2$-bundle can be obtained from an $SL(2;\Z)$-bundle by some $T^2$-surgery. 
In the subsequent sections, 
we always denote an $SL(2;\Z)$-bundle by $M(A_1, B_1, \ldots , A_g, B_g)$ omitting $(m, n)=(0,0)$. \\

Now we are ready to complete the proof of Theorem~\ref{isom}. \\

\begin{prop}~\label{bundle iso}
The two orientable $T^2$-bundles 
$$\pi \colon M(A_1,B_1, \ldots , A_g, B_g; m, n)\to \Sigma _g, \;\; 
\pi ' \colon M(A_1,B_1, \ldots , A_g, B_g; k, l)\to \Sigma _g$$ 
are isomorphic if there exist $\u _i, \v _i\in \Z ^2$ such that 
$$\begin{pmatrix}k\\l\end{pmatrix}=\begin{pmatrix}m\\n\end{pmatrix}+\sum_{i=1}^g \left((A_i-E_2)\u_i +(B_i-E_2)\v_i\right).$$
\end{prop}
\begin{proof}
We explicitly construct a bundle isomorphism 
$$\widetilde{f}\colon M(A_1, B_1, \ldots , A_g, B_g; k, l) \to M(A_1, B_1, \ldots , A_g, B_g; m, n).$$
In the following, we assume $g\geq 2$, since the claim is trivial when $g=0$ 
and the construction has been done in Proposition 2 (3) of \cite{SF83} when $g=1$. 
The idea of our construction is as follows. 
First we construct a bundle isomorphism over the $1$-skeleton of the standard cellular decomposition of $\Sigma_g$ 
so that it coincides with $\widetilde{f}_{B_i\v_i}\colon M(A_i)\to M(A_i)$ over $e^1_{2i-1}$, 
and with $\widetilde{f}_{-\u_i}\colon M(B_i)\to M(B_i)$ over $e^1_{2i}$. 
Then it extends over the $2$-cell $e^2$ to become a bundle isomorphism $\widetilde{f}$ over $\Sigma _g$. 

Let us write down $\widetilde{f}$ explicitly by using the above model of orientable $T^2$-bundles. 
First we take a fundamental domain $P$ of the action of $\pi _1(\Sigma_g)$ 
which is a hyperbolic regular $4g$-gon centered at $0\in \mathbb{D}$. 
Notice that each vertex of $P$ corresponds to the base point $b_0$ of $\Sigma_g$ and the boundary $\partial P$ is written 
$$\partial P
=[\alpha _1, \beta _1][\alpha _2, \beta _2]\cdots [\alpha _g, \beta _g]
=\alpha _1\beta _1\alpha _1^{-1}\beta _1^{-1}\alpha _2\beta _2\alpha _2^{-1}\beta _2^{-1}\cdots \alpha _g\beta _g\alpha _g^{-1}\beta _g^{-1}
$$
as an element of $\pi _1(\Sigma_g, b_0)$. 
We take a parametrization $w\colon [0, 4g]\to \partial P\cong S^1$ so that for each integer $4i+j$ ($0\leq i\leq g, \;j=0,1,2,3$), 
the interval $[4i+j, 4i+j+1]$ is mapped to $\alpha _i$, $\beta _i$, $\alpha _i^{-1}$, $\beta _i^{-1}$ according as $j=0,1,2,3$, and that 
\begin{eqnarray*}
w(u) \cdot \beta _i&=&w(8i+3-u) \;\;\;\; (4i\leq u\leq 4i+1),\\ 
w(u) \cdot \alpha _i^{-1}&=&w(8i+5-u) \;\;\;\; (4i+1\leq u\leq 4i+2).
\end{eqnarray*}
Now we define 
$$\widetilde{f}\colon \pi ^{-1}(\partial P) \to {\pi '}^{-1}(\partial P)$$ by 
\begin{eqnarray*}
\widetilde{f}
\begin{pmatrix}
\begin{bmatrix}
w(u), \begin{bmatrix}
    s \\ t
\end{bmatrix}
\end{bmatrix}
\end{pmatrix}
&=&
\begin{bmatrix}
w(u), \begin{bmatrix}
\begin{pmatrix}
    s \\ t
\end{pmatrix}+(u-4i)B_i\v _i
\end{bmatrix}
\end{bmatrix} \;\;\;\;\;\;\;\;
(4i\leq u\leq 4i+1), \\
& &
\begin{bmatrix}
w(u), \begin{bmatrix}
\begin{pmatrix}
    s \\ t
\end{pmatrix}+(4i+2-u)\u _i
\end{bmatrix}
\end{bmatrix} 
\;\;\;\;\;\; 
(4i+1\leq u\leq 4i+2), \\
& &
\begin{bmatrix}
w(u), 
\begin{bmatrix}
\begin{pmatrix}
    s \\ t
\end{pmatrix}+(4i+3-u)\v _i
\end{bmatrix}
\end{bmatrix}
\;\;\;\;\;\;
(4i+2\leq u\leq 4i+3), \\
& &
\begin{bmatrix}
w(u), 
\begin{bmatrix}
\begin{pmatrix}
    s \\ t
\end{pmatrix}+(u-4i-3)A_i\u _i
\end{bmatrix} 
\end{bmatrix}
\;\; 
(4i+3\leq u\leq 4i+4).
\end{eqnarray*}
This induces a bundle isomorphism over the $1$-skeleton and, 
moreover, specifies a bundle isomorphism over $\partial P$,  
which corresponds to an element of $\pi _1(\Diff_+(T^2))=\Z^2$. 
In fact, this element is exactly $(k-m, l-n)$, since 
$$\sum_{i=1}^g (B_i\v _i-\u _i-\v _i+A_i\u _i)
=\sum_{i=1}^g \left((A_i-E_2)\u_i +(B_i-E_2)\v_i\right)
=\begin{bmatrix}
    k-m \\ l-n
\end{bmatrix}$$ by assumption. 
This resolves the difference in the coefficients of the $T^2$-surgeries corresponding to the two bundles $\pi$ and $\pi'$.
Consequently, $\widetilde{f}$ extends in a natural way over the $2$-cell $P$, giving rise to a bundle isomorphism that covers the identity map of $\Sigma_g$. 
\end{proof}

\subsection{Fiber connected sum}
In this section, we explain the fiber connected sum, 
a method of constructing a new $T^2$-bundle from two $T^2$-bundles. 
Let $\xi =(M, \pi , \Sigma _g, T^2)$ and $\xi'=(M', \pi ', \Sigma _{g'}, T^2)$ be orientable $T^2$-bundles. 
Let $b_0$ (resp. $b_0'$) be a base point on the surface $\Sigma _g$ (resp. $\Sigma _{g'}$). 
We take local trivializations of $\xi $ and $\xi'$ around each base point, and denote them by 
\[ \phi _U\colon \pi^{-1}(U)\to U\times T^2, \;\; 
{\phi' }_{U'}\colon \pi^{-1}(U')\to U'\times T^2, \] respectively. 
Now we take $2$-disks $D$ and $D'$ so that $b_0\in D\subset U$ and $b_0'\in D'\subset U'$. 
Then $\partial D$ and $\partial D'$ carry natural orientations derived from those of $\Sigma _g$ and $\Sigma_{g'}$, respectively. 
We take an orientation reversing diffeomorphism $f\colon \partial D\to \partial D'$ 
and an orientation preserving diffeomorphism $h \colon T^2\to T^2$. 

\begin{figure}[h]
     \centering
     \includegraphics[width=11cm, pagebox=cropbox,clip]{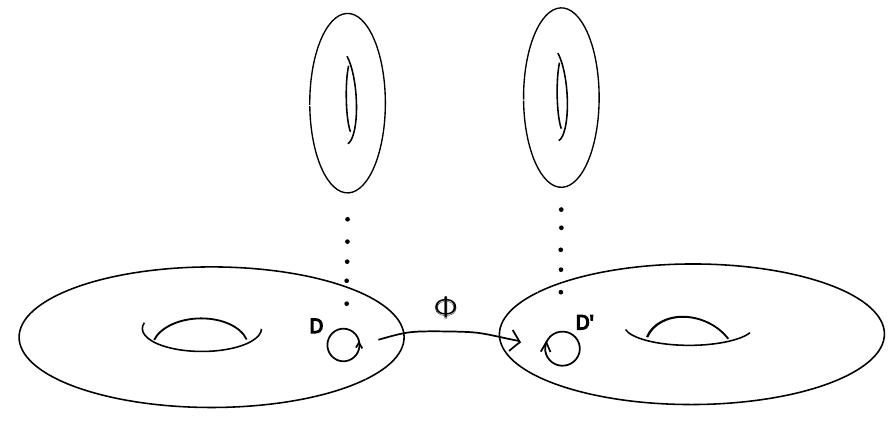}
     \caption{Fiber connected sum}
     
 \end{figure}
Then we obtain an orientable $T^2$-bundle over $\Sigma_{g+g'}$ from two bundles 
$\pi | \colon \pi ^{-1}(\Sigma _g\setminus D)\to \Sigma _g\setminus D$ and  
$\pi '| \colon {\pi '}^{-1}(\Sigma _{g'}\setminus D')\to \Sigma _{g'}\setminus D'$ 
by the fiberwise gluing along their boundaries given by 
$$\Phi\colon \pi^{-1}(\partial D)\xrightarrow{\phi_U} \partial D\times T^2\xrightarrow{f\times h} \partial D'\times T^2\xrightarrow {({\phi '}_{U'})^{-1}}\pi'^{-1}(\partial D'). $$
This is called the fiber connected sum of $\xi $ and $\xi'$, 
and denoted by $\xi \# \xi '$ or $$\pi \# \pi '\colon M\#_{f}M' \to \Sigma _{g+g'}.$$
By this definition, however, the diffeomorphism type of $M\#_{f}M'$ depends on 
the choices of the local trivializations $\phi _U$, ${\phi' }_{U'}$ and the diffeomorphism $h$. 
In order to avoid such an ambiguity, we fix $\phi _U$, ${\phi' }_{U'}$ and $h$ as follows. 


First we describe the two $T^2$-bundles $\xi $ and $\xi'$ by 
\[ \pi \colon M(A_1, B_1, \ldots , A_g, B_g; m, n)\to \Sigma _g, \;\; 
\pi ' \colon M(C_1, D_1, \ldots , C_{g'}, D_{g'}; k, l)\to \Sigma _{g'}, \]
and take the local trivializations $\phi _U$ and $\phi'_{U'}$ 
so that ${\phi _U} |_{F_0}=\Psi ^{-1}$ and $\phi '_{U'} |_{F'_0}={\Psi' }^{-1}$,  
where $\Psi \colon T^2\to F_0=\pi ^{-1}(b_0)$ and $\Psi'\colon T^2\to F'_0=\pi '^{-1}(b'_0)$ 
are the diffeomorphisms fixed in the definition of the monodromies of $\pi $ and $\pi'$. 
Now we fix $h=\mathrm{id}_{T^2}$, then the fiber connected sum of $\xi $ and $\xi'$ can be written as 
\[ \pi \# \pi' \colon M(A_1, B_1, \ldots , A_g, B_g, C_1, D_1, \ldots , C_{g'}, D_{g'}; m+k, n+l) 
\to \Sigma _{g+g'}. \] 
In particular, 
$$M(A, B; m, n)\#_f M(C, D; k, l)\cong M(A, B, C, D; m+k, n+l)$$
holds when $g=g'=1$. 
Thus the monodromy and the Euler class of the $T^2$-bundle $\xi \# \xi'$ 
is described by those of $\xi $ and $\xi'$. 

\subsection{Decomposition theorem}

Recall that any orientable $T^2$-bundle over $\Sigma _g$ can be described as 
$$\pi \colon M(A_1, B_1, \ldots , A_g, B_g; m, n)\to \Sigma _g$$ 
by $(m, n)\in \Z^2$ and $A_1, B_1, \ldots , A_g, B_g\in SL(2;\Z)$ with $[A_1, B_1]\cdots [A_g, B_g]=E_2$. 
Then, applying Corollary~\ref{sl} and Proposition~\ref{isom} to the monodromy $\widetilde{\rho }$ of $\pi $, we obtain the following. 

\vspace{8pt}

\begin{theorem}~\label{thm:pre-classification}
Let $\xi $ be an orientable $T^2$-bundle over $\Sigma _g$. 
Then there exist $B_1, \ldots , B_g\in SL(2;\Z)$ and $(m,n)\in \Z^2$ such that 
$\xi $ is isomorphic to either of the $2^g$ $T^2$-bundles $M(\pm E_2, B_1, \ldots , \pm E_2, B_g; m, n)$. 
\end{theorem}

Moreover, the following is obtained as a corollary. 

\vspace{8pt}

\begin{theorem}~\label{thm:decomp}
Any orientable $T^2$-bundle over $\Sigma _g$ with $g\geq 1$ is isomorphic to 
the fiber connected sum of $g$ pieces of $T^2$-bundles over $T^2$. 
\end{theorem}

\begin{proof}
Since $\pm E_2$ commute with any element of $SL(2;\Z)$, 
$$M(\pm E_2, B_1; m, n), \; M(\pm E_2, B_2; 0,0), \; \cdots , \; M(\pm E_2, B_g; 0, 0)$$
are all orientable $T^2$-bundles over $T^2$. 
By the consideration in \S 4.4, it follows that 
\begin{eqnarray*}
M(\pm E_2, B_1; m, n) \#_f M(\pm E_2, B_2; 0,0) \#_f \cdots \#_f M(\pm E_2, B_g; 0, 0)\\
\cong M(\pm E_2, B_1, \pm E_2, B_2, \ldots , \pm E_2, B_g; m, n). 
\end{eqnarray*}
Therefore, by Theorem~\ref{thm:pre-classification}, any orientable $T^2$-bundle over $\Sigma _g$ is 
isomorphic to the fiber connected sum of $g$ pieces of $T^2$-bundles over $T^2$. 
\end{proof}

\section{Classification of orientable $T^2$-bundles over $\Sigma _g$}
In this section, we prove the classification theorem of orientable $T^2$-bundles over $\Sigma _g$
(Theorem~\ref{main thm}), which is our main theorem. Since the classification is already known for the case $g=0, 1$ (\S5.1, 5.2), the main part is the case $g\geq 2$ (\S5.3). There, we first classify the $\mathcal{M}_g$-orbits of $\mathrm{Hom}(\pi _1(\Sigma _g), SL(2; \Z))$ (Theorem~\ref{hom-sl}). 
This is equivalent, up to the conjugate action of $SL(2; \Z)$, to the classification of isomorphism classes of $SL(2;\Z)$-bundles over $\Sigma _g$. Hence, we can complete the proof of the main theorem by adding a simple argument about Euler classes to it. 

\vspace{5pt}

\subsection{The case $g=0$}~\label{genus 0}
Since $S^2$ is simply-connected, any $T^2$-bundle over $S^2$ has trivial monodromy, 
and hence, admits a principal $T^2$-bundle structure. 
The isomorphism classes of principal $T^2$-bundles are classified by their Euler classes 
$$(m, n)\in \pi_2(BT^2)=\pi _2(BS^1\times BS^1)=\pi _2(\C P^{\infty }\times \C P^{\infty })=\Z^2,$$ 
but as orientable $T^2$-bundles, $M(m, n)$ and $M(d,0)$ are isomorphic by Remark~\ref{rem:Euler}, 
where $d$ is the greatest common divisor of $m$ and $n$. 
Thus the isomorphism classes of orientable $T^2$-bundles are classified by $d\in \Z_{\geq 0}$. 

\subsection{Theorem of Sakamoto and Fukuhara (the case $g=1$)}

The isomorphism classes of orientable $T^2$-bundles over $T^2$ have been completely classified as follows. 

\vspace{8pt}

\begin{theorem}[Sakamoto-Fukuhara \cite{SF83}]~\label{SF}
Let $\xi =(M^4, \pi , T^2, T^2)$ be an orientable $T^2$-bundle over $T^2$. 
Then there exist $B\in SL(2; \Z)$ and $(m, n)\in \Z^2$ such that 
$\xi $ is isomorphic to either $M(E_2, B; m, n)$ or $M(-E_2, B; m, n)$. 
Moreover, two orientable $T^2$-bundles $M(\ve E_2, B; m, n)$ and $M(\delta E_2, C; k, l)$ 
are isomorphic if and only if one of the following conditions is satisfied, where $\ve, \delta =\pm 1$. 
\begin{enumerate}
\item
$\ve=\delta =1$ and there exist $Q\in SL(2; \Z)$ and $\x \in \Z^2$ such that $QCQ^{-1}=B^{\pm 1}$ and 
\[\begin{pmatrix}m\\n \end{pmatrix}-Q\begin{pmatrix}k\\l\end{pmatrix}=(B-E_2){\x}. \]
\item
$\ve =1$, $\delta =-1$, $\ord (B)$ is either $2$, $4$ or $6$ and there exist $Q\in SL(2; \Z)$ and $\x \in \Z^2$ such that $QCQ^{-1}=\pm B^{\pm 1}$ and 
\[\begin{pmatrix}m\\n \end{pmatrix}-Q\begin{pmatrix}k\\l\end{pmatrix}=(B-E_2){\x}. \]
\item
$\ve =-1$, $\delta =1$, $\ord (C)$ is either $2$, $4$ or $6$ and there exist $Q\in SL(2; \Z)$ and $\x \in \Z^2$ such that $QBQ^{-1}=\pm C^{\pm 1}$ and 
\[\begin{pmatrix}k\\l \end{pmatrix}-Q\begin{pmatrix}m\\n\end{pmatrix}=(C-E_2){\x}.\]
\item
$\ve=\delta =-1$ and there exist $Q\in SL(2; \Z)$ and $\x, \y \in \Z^2$ 
such that $QCQ^{-1}=\pm B^{\pm 1}$ and 
\[\begin{pmatrix}m\\n \end{pmatrix}-Q\begin{pmatrix}k\\l\end{pmatrix}=(B-E_2)\x +2\y .\]
\end{enumerate}
\end{theorem}

The former claim is nothing but a special case of Theorem~\ref{thm:pre-classification}. 
On the other hand, we need another lemma for the proof of the latter claim. 
As is well known, the order of $B$ in $SL(2; \Z)$ is either $1$, $2$, $3$, $4$, $6$ or $\infty $. 
Hence, the cyclic subgroup $\left<B\right>$ of $SL(2; \Z)$ generated by $B$ contains $-E_2$ 
if and only if $\ord (B)$ is either $2$, $4$ or $6$. 
In these cases, the following holds. 

\vspace{8pt}

\begin{prop}~\label{finite order}
If the order of $B\in SL(2;\Z)$ is equal to $4$, 
all the four $T^2$-bundles $M(\pm E_2, \pm B; m,n)$ are isomorphic to each other. 
If it is equal to $2$ or $6$, the following holds; 
\begin{eqnarray*}
M(E_2, B; m,n)\cong M(-E_2, B; m,n)\cong M(-E_2, -B; m,n) \not \cong M(E_2, -B; m,n), 
\end{eqnarray*}
where $\cong $ and $\not \cong $ mean ``isomorphic" and ``non-isomorphic" as $T^2$-bundles, respectively. 
\end{prop}
\begin{proof}
In each case, we will construct a bundle isomorphism using Corollary~\ref{g=1}. 
Since $(m, n)=(m', n')$, we may assume that $Q=E_2$ and $\x=\y={\bf 0}$. 
Thus, what we only have to do is to find out an appropriate $P\in SL(2; \Z)$ for each case. 
\begin{enumerate}
\item
The case where $\mathrm{ord} (B)=4$. 
Since we have $B^2=-E_2$, 
an isomorphism $$M(E_2, B; m,n)\cong M(-E_2, B; m,n)$$ is obtained 
by setting $P=\begin{pmatrix}1&0\\2&1\end{pmatrix}$. 
Similarly, setting $P=\begin{pmatrix}-1&-1\\4&3\end{pmatrix}, \begin{pmatrix}1&1\\2&3\end{pmatrix}$, 
we obtain isomorphisms 
$$M(E_2, B; m,n)\cong M(E_2, -B; m,n), \;\; M(E_2, B; m,n)\cong M(-E_2, -B; m,n),$$ respectively. 
\item
The case where $\mathrm{ord} (B)=6$. 
Since we have $B^3=-E_2$, 
we obtain isomorphisms 
$$M(E_2, B; m,n)\cong M(-E_2, B; m,n), \;\; M(E_2, B; m,n)\cong M(-E_2, -B; m,n)$$ 
by setting $P=\begin{pmatrix}1&0\\3&1\end{pmatrix}, \begin{pmatrix}1&1\\3&4\end{pmatrix}$, respectively. 
On the other hand, since $-E_2\not \in \left< -B\right>$ and $-E_2\in \left<B\right>$, we have 
$$\left< E_2, -B\right>=\left< -B\right>\ne \left<B\right>=\left<E_2, B\right>.$$
Then it follows that \[ M(E_2, B; m,n)\not \cong M(E_2, -B; m,n), \] 
since their monodromy groups are different.  
\item
The case where $\mathrm{ord} (B)=2$, that is, $B=-E_2$. 
In this case, we obtain isomorphisms 
$$M(E_2, -E_2; m,n)\cong M(-E_2, -E_2; m,n), \; M(E_2, -E_2; m,n)\cong M(-E_2, E_2; m,n)$$
by setting $P=\begin{pmatrix}1&0\\1&1\end{pmatrix}, \begin{pmatrix}0&-1\\1&0\end{pmatrix}$, respectively. On the other hand, the two bundles $M(E_2, -E_2; m,n)$ and $M(E_2, E_2; m,n)$ are not isomorphic to each other, since they have different monodromy groups. 
\end{enumerate}
\end{proof}

By using this proposition, we prove the theorem of Sakamoto and Fukuhara. 

\begin{proof}[{\bf Proof of Theorem~\ref{SF}}]
The former claim of Theorem~\ref{SF} follows by putting $g=1$ in Theorem~\ref{thm:pre-classification}. 

We prove the latter claim by using Corollary~\ref{g=1}. 
First we deal with the cases (2) and (3), 
but it suffices to argue about (2) since (3) can be replaced by (2) by exchanging the roles of $B$ and $C$. 
The images of the monodromies of $M(E_2, B; m, n)$ and $M(-E_2, C; k, l)$ are 
$\left<B\right>$ and $-E_2\in \left<-E_2, C\right>$, respectively. 
Hence, it is necessary for the two bundles to be isomorphic 
that these subgroups of $SL(2; \Z)$ coincide up to conjugation. 
In particular, $\left<B\right>$ must contain $-E_2$. 
Therefore, the order of $B$ should be $2$, $4$ or $6$. 
Moreover, there exists $Q\in SL(2;\Z)$ such that $QCQ^{-1}=\pm B^{\pm 1}$, 
since $C$ must coincide with a power of $B$ up to conjugation. 
Applying Proposition~\ref{finite order}, we can find an appropriate $P\in SL(2;\Z)$ 
for each case so that the monodromies of the two bundles coincide. 
Finally, we consider about the Euler class. 
Since the Euler class of $M(E_2, B; m, n)$ belongs to the $\Z$-module 
$$H^2\left(\Sigma _g; \{ \pi _1(T^2) \}\right)\cong \Z^2/\left< (B-E_2)\e_1, (B-E_2)\e _2\right>, $$
the condition that the Euler classes of the two bundles coincide can be written as 
\[ \begin{pmatrix}m\\n\end{pmatrix}-Q\begin{pmatrix}k\\l\end{pmatrix} \in \left< (B-E_2)\e_1, (B-E_2)\e _2\right>=\{ (B-E_2)\x \mid \x\in \Z^2 \}. \] 
This completes the proof of the cases (2) and (3). \par 

What is left to us is to deal with the cases (1) and (4), but they are rather simple.  
If the monodromies of the two bundles coincide up to conjugation, 
then there exists $Q\in SL(2;\Z)$ satisfying $QCQ^{-1}=B^{\pm 1}$ in the case of (1), 
and $QCQ^{-1}=\pm B^{\pm 1}$ in the case of (4). 
Conversely, in all these cases, we can easily find $P\in SL(2;\Z)$ 
that makes the monodromies of the two bundles coincide. 
Finally, the Euler class of the $T^2$-bundle belongs to the $\Z$-module 
\[ H^2\left(\Sigma _g; \{ \pi _1(T^2) \}\right)\cong \Z^2/\left< (B-E_2)\e_1, (B-E_2)\e _2\right> \] 
in the case of (1), and to 
\[ H^2\left(\Sigma _g; \{ \pi _1(T^2) \}\right)\cong \Z^2/\left< 2\e_1, 2\e_2, (B-E_2)\e_1, (B-E_2)\e _2\right> \]
in the case of (4). 
Then we can write down the condition that the Euler classes of the two bundles coincide 
as that on integers $k$, $l$, $m$, $n$. 
Thus we have proven the latter claim. 
\end{proof}

\subsection{Main Theorem (the case $g\geq 2$)}

When $g\geq 2$, the classification of monodromies becomes more complicated 
since the fundamental group $\pi _1(\Sigma _g)$ is non-commutative. 
Based on Theorem~\ref{hom-psl}, however, 
$\mathcal{M}_g$-orbits of $\mathrm{Hom}(\pi_1(\Sigma _g), SL(2;\Z))$ can be classified (Theorem~\ref{hom-sl}), and then, we obtain the main theorem (Theorem~\ref{main thm}). 
Before going into the detail of the arguments, we first prepare some propositions needed later. 

\vspace{8pt}

\begin{prop}~\label{key prop}
For any $B, C\in SL(2;\Z)$, the following isomorphism holds; 
$$M(-E_2, B, E_2, C)\cong M(-E_2, B, E_2, -C).$$
\end{prop}
\begin{proof}
Let $f\colon \Sigma_2\to \Sigma_2$ be the product of the right-handed Dehn twists about $a$ and $c$ and the left-handed Dehn twist about $b$, where $a$, $b$ and $c$ are the simple closed curves depicted in Figure~\ref{fig:curves}. 
The homomorphism $f_{\ast }\colon \pi _1(\Sigma _2)\to \pi _1(\Sigma _2)$ satisfies 
$$f_{\ast }(\alpha _1)=\alpha _1, \; f_{\ast }(\beta _1)=\alpha _1^{-1}\alpha _2^{-1}\beta _1\alpha _1, \; 
f_{\ast }(\alpha _2)=\alpha _2, \; f_{\ast }(\beta _2)=\alpha _2^{-1}\alpha _1^{-1}\beta _2\alpha _2. $$
\begin{figure}
\begin{center}
\includegraphics[width=7cm]{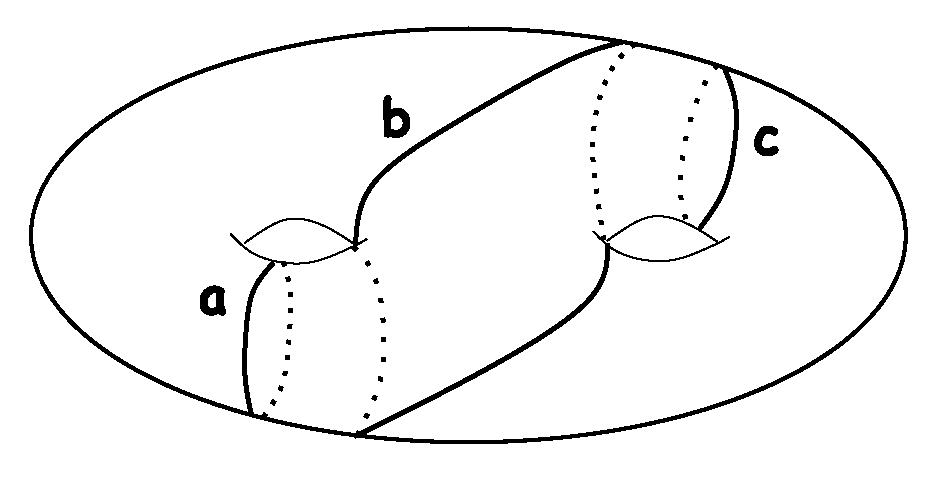}
\caption{Simple closed curves $a$, $b$ and $c$}
\label{fig:curves}
\end{center}
\end{figure}
Hence we have 
$$(\tilde{\rho} \circ f_{\ast})(\alpha _1)=-E_2, \; (\tilde{\rho} \circ f_{\ast})(\beta _1)=B, \; 
(\tilde{\rho} \circ f_{\ast})(\alpha _2)=E_2, \; (\tilde{\rho} \circ f_{\ast})(\beta _2)=-C.$$
Therefore, pulling back the $T^2$-bundle $\pi \colon M(-E_2, B, E_2, C)\to \Sigma _2$ by $f$, 
we obtain the $T^2$-bundle $f^{\ast }\pi \colon M(-E_2, B, E_2, -C)\to \Sigma _2$. 
Since $f$ is a self-diffeomorphism of $\Sigma _2$, we have an isomorphism $f^{\ast }\pi \cong \pi $.  
\end{proof}

For any homomorphism $\rho \colon \pi_1(\Sigma _g)\to PSL(2;\Z)$, there are $2^{2g}$ choices of taking a lift $\widetilde{\rho }\colon \pi_1(\Sigma _g)\to SL(2;\Z)$ of $\rho $ with respect to the projection $p\colon SL(2;\Z)\to PSL(2;\Z)$. We want to determine which two lifts are transformed to each other by the action of the mapping class group $\mathcal{M}_g$. 
From now on, we suppose that $\rho $ is of the normal form in the sense of Theorem~\ref{hom-psl}. 
Put $\widetilde{\rho}(\alpha _i)=A_i$ and $\widetilde{\rho}(\beta _i)=B_i$. 
Then we have $A_i=E_2$ or $A_i=-E_2$ for each $i$ with $1\leq i\leq g$, so this can be described as $A_i=\ve _iE_2$, where $\ve _i$ is equal to either $+1$ or $-1$ for each $i$. 

\vspace{8pt}

\begin{prop}~\label{2,4,6}
The monodromy group of $M(\ve _1E_2, B_1, \ldots , \ve _gE_2, B_g)$ does not contain $-E_2$ 
if and only if $\ve _i=1$ and the order of $B_i$ is neither $2$, $4$ nor $6$ for each $i$ with $1\leq i\leq  g$.
\end{prop}
\begin{proof}
The only if part is obvious. 
On the other hand, the if part follows from the condition that $\rho $ is of the normal form. 
The proof is as follows. 
Since $\ord (B_i)\ne 2$, we have $B_i=E_2$ for each $i$ with $m+1\leq i\leq g$. 
Now suppose that the monodromy group $\mathrm{Im}(\widetilde{\rho })$ contains $-E_2$. 
Then we can write $-E_2=B_{\sigma_1}^{n_1}B_{\sigma_2}^{n_2}\cdots B_{\sigma_r}^{n_r}$, 
where $r$ is a positive integer and $\sigma _i \in \{1, 2, \ldots ,m\}$ for each $i$. 
Then we have 
$$p (B_{\sigma_1})^{n_1}p (B_{\sigma_2})^{n_2}\cdots p (B_{\sigma_r})^{n_r}=e\in \mathrm{Im}(\rho )\subset PSL(2;\Z).$$
Since each $p (B_i)$ ($1\leq i\leq m$) is a generator of the factor $H_i$ of the free product decomposition of $\mathrm{Im}(\rho )$, it follows that $n_1=n_2=\cdots =n_r=0$. 
Thus we have $E_2=-E_2$, which is a contradiction. 
Therefore, $\mathrm{Im}(\widetilde{\rho })$ does not contain $-E_2$. 
\end{proof}

We note that the monodromy group $\mathrm{Im}(\widetilde{\rho })$ contains an element of order $4$ 
if and only if $l<m$, since the condition that $B\in SL(2;\Z)$ is of order $4$ 
is equivalent to the one that $p (B)\in PSL(2;\Z)$ is of order $2$. 
Moreover, if $l<m$, then we have $\ord (B_i)=4$ for any $i$ with $l+1\leq i\leq m$. 
Then, applying Proposition~\ref{finite order} and using the isomorphisms 
$$M(E_2, B_i)\cong M(-E_2, B_i)\cong M(E_2, -B_i)\cong M(-E_2, -B_i),$$ 
we can freely switch $\ve _i$ and the sign of $B_i$ 
without changing the isomorphism class of the $T^2$-bundle $M(\ve _1E_2, B_1, \ldots , \ve _gE_2, B_g)$. 
Similarly, we can exclude the elements of order $2$ and $6$ from $B_1, \ldots , B_g$ preserving the isomorphism class of the $T^2$-bundle. 
For, if $B_i$ is of order $2$ or $6$, then $B_i$ can be replaced by $-B_i$ (of order $1$ or $3$) by using the bundle isomorphisms $$M(E_2, B_i)\cong M(-E_2, -B_i), \; M(-E_2, B_i)\cong M(-E_2, -B_i).$$
Therefore, we don't have to deal with all the $2^{2g}$ possibilities of lifts of $\rho $, 
but only those satisfying the following three conditions; 
\begin{enumerate}
\item[(a)]
if $k+1\leq i\leq l$, then $\ord(B_i)=3$,  
\item[(b)]
if $l+1\leq i\leq m$, then $\ve _i=-1$, and 
\item[(c)]
if $m+1\leq i\leq g$, then $B_i=E_2$. 
\end{enumerate}
In what follows, we assume that 
the monodromy $\widetilde{\rho }$ of $M(\ve _1E_2, B_1, \ldots , \ve_gE_2, B_g)$ satisfies these conditions. 

\vspace{8pt}

\begin{prop}~\label{cases}
The following assertions hold. 
\begin{enumerate}
\item
When the monodromy group of $M(\ve _1E_2, B_1, \ldots , \ve _gE_2, B_g)$ contains $-E_2$, all the $2^g$ $T^2$-bundles $M(\ve _1E_2, \pm B_1, \ldots , \ve _gE_2, \pm B_g)$ are isomorphic to each other. 
\item
When the monodromy group of $M(E_2, B_1, \ldots , E_2, B_g)$ does not contain $-E_2$, 
the $2^k$ $T^2$-bundles $M(E_2, \pm B_1, \ldots , E_2, \pm B_k, E_2, B_{k+1}, \ldots, E_2, B_g)$ 
are pairwise non-isomorphic. 
\end{enumerate}
\end{prop}
\begin{proof}
\begin{enumerate}
\item
By assumption, there exists $i$ with $1\leq i\leq g$ such that $\ve _i=-1$. 
For, if $\ve _i=1$ for all $i$, then by Proposition~\ref{2,4,6} and conditions (a) and (c), 
there must be $B_i$ of order $4$, and hence, we have $\ve _i=-1$ by the condition (b), 
which contradicts to the assumption that $\ve _i=1$ for all $i$. 
Hence, we can take $i$ ($1\leq i\leq g$) with $\ve _i=-1$. 
Then, it follows from the isomorphisms $M(-E_2, B_i)\cong M(-E_2, -B_i)$ and  
$$M(-E_2, B_i)\#_f M(E_2, B_j)\cong M(-E_2, B_i)\#_f M(E_2, -B_j)$$ 
(see Propositions~\ref{finite order} and~\ref{key prop}) that replacing $B_j$ by $-B_j$ ($1\leq j\leq g$) does not change the isomorphism class of $T^2$-bundle. 
\item
Let us denote the monodromy group $\widetilde{\rho } (\pi _1(\Sigma _g))\subset SL(2; \Z)$ by $H$. Since $H$ does not contain $-E_2$, it is impossible for $H$ to contain both $\pm B_i$ for any $i$ with $1\leq i\leq k$. 
Hence, by replacing $B_i$ by $-B_i$ or vice versa, we can switch whether $B_i\in H$ or $-B_i\in H$ 
preserving the condition that $H$ does not contain $-E_2$. 
Therefore, the $2^k$ $T^2$-bundles $$M(E_2, \pm B_1, \ldots , E_2, \pm B_k, E_2, B_{k+1}, \ldots, E_2, B_g)$$ 
are indeed pairwise non-isomorphic, since they have different monodromy groups. 
\end{enumerate}
\end{proof}

\vspace{8pt}

\begin{prop}~\label{infinite order}
The two $T^2$-bundles 
\begin{eqnarray*}
M(\ve_1E_2, B_1, \ldots , \ve_kE_2, B_k, \ve_{k+1}E_2, B_{k+1}, \ldots, \ve_gE_2, B_g) \\
M(\delta_1E_2, B_1, \ldots , \delta_kE_2, B_k, \delta_{k+1}E_2, B_{k+1}, \ldots, \delta_gE_2, B_g)
\end{eqnarray*}
are not isomorphic to each other if $(\ve_1, \ldots , \ve_k)\ne (\delta _1, \ldots , \delta _k)$. 
\end{prop}
\begin{proof}
We will prove the assertion by reduction to the absurd, that is, we suppose $(\ve_1, \ldots , \ve_k)\ne (\delta _1, \ldots , \delta _k)$ and the two bundles are isomorphic, and then lead a contradiction.  
Without loss of generality we may assume that $\ve_1=1$ and $\delta _1=-1$, 
so it is enough to consider the case where there exists a bundle isomorphism 
\begin{eqnarray*}
\tilde{\phi }\colon M(E_2, B_1, \ve_2E_2, B_2, \ldots, \ve_gE_2, B_g) \to M(-E_2, B_1, \delta_2E_2, B_2, \ldots, \delta_gE_2, B_g) 
\end{eqnarray*}
covering the base diffeomorphism $\phi \colon \Sigma_g\to \Sigma_g$. 
Denote the monodromy representation of  
$M(-E_2, B_1, \delta_2E_2, B_2, \ldots, \delta_gE_2, B_g)$ by $\rho $. 
Then we have  
\[  (\rho\circ \phi _{\ast })(\alpha _1)=E_2, \;\;
(\rho\circ \phi _{\ast })(\alpha _i)=\ve_{i}E_2 \;\; (2\leq i \leq g), \;\;  
(\rho\circ \phi _{\ast })(\beta _j)=B_{j}, \]
where $\alpha _j, \beta _j$ ($1\leq j\leq g$) are the canonical generators of $\pi _1(\Sigma _{g})$. 
Then, for each $i$ ($2\leq i\leq g$), the loops $\phi (\alpha _i)$, $\phi (\beta _i)$ are contained in the smallest normal subgroup $N$ generated by ${\alpha}_1, {\alpha} _2, \ldots, {\alpha}_g, {\beta }_2, \ldots , {\beta }_g$, because $\rho (\phi (\alpha _i))=\ve_{i}E_2$, $\rho (\phi (\beta _i))=B_{i}$ are contained in $\rho (N)=\langle -E_2, B_2, \ldots, B_g \rangle ^{\Im (\rho )}$ and the homomorphism 
\[ \bar{\rho } \colon \pi_1(\Sigma _g)/N\cong \langle {\beta }_1\rangle \cong \Z\to 
\rho \left(\pi_1(\Sigma _g)\right)/\rho (N) \cong \left<B _1\right>\cong \Z \]
induced by $\rho \colon \pi_1(\Sigma _g)\to SL(2; \Z)$ is an isomorphism. 
Here we note that $B_1$ does not belong to $\rho (N)$, since $p\circ \rho $ is of normal form and $p(B_1)$ is a generator of one factor of the free product decomposition of $\Im (p\circ \rho )$, where $p\colon SL(2;\Z)\to PSL(2;\Z)$ is the canonical projection. 
Since the loops $\phi (\alpha _i)$ and $\phi (\beta _i)$ ($2\leq i \leq g$) are elements of $N$, 
their intersection numbers with ${\alpha }_1$ are all $0$. 
Hence, the loop $\phi^{-1}(\alpha _1)$ has intersection number $0$ with any of $\alpha _i$ and $\beta _i$ ($2\leq i \leq g$). 
Recalling that the first homology group is the abelianization of the fundamental group, 
$\phi^{-1}(\alpha _1)$ is homologous to an integral linear combination of $\alpha _1$ and $\beta _1$. 
Since $p\circ \rho\circ \phi _{\ast }=p\circ \rho $ is of the normal form 
and $ (\rho\circ \phi _{\ast })(\phi ^{-1}(\alpha _1))=\rho (\alpha _1)=-E_2$, the coefficient of $\beta _1$ must be $0$. 
Thus we have $$\phi^{-1}(\alpha _1)=\alpha _1^ns, $$ 
where $n$ is an integer and $s$ is an element of the commutator subgroup $[\pi _1(\Sigma_g), \pi _1(\Sigma_g)]$. 
Then we obtain that 
$$(\rho\circ \phi _{\ast })(s)=(\rho\circ \phi _{\ast })(\alpha _1^{-n}\phi ^{-1}(\alpha _1))
=(\rho\circ \phi _{\ast })(\alpha _1)^{-n}\rho(\alpha _1)=-E_2.$$
This implies that $-E_2$ belongs to the commutator subgroup $[SL(2; \Z), SL(2; \Z)]$, which is a contradiction. 
Thus we have obtained that 
\[M(E_2, B_1, \ve_2E_2, B_2, \ldots, \ve_gE_2, B_g)
\not \cong M(-E_2, B_1, \delta_2E_2, B_2, \ldots, \delta_gE_2, B_g). \] 
This completes the proof. 
\end{proof}

Based on these propositions, we can prove Theorem~\ref{hom-sl} as follows. 

\begin{proof}[{\bf Proof of }$\mathrm{{\bf Theorem~\ref{hom-sl}}}$]
\begin{enumerate}
\item
The case where $m>l$. 
For any lift $\widetilde{\rho }$ of $\rho$, 
the monodoromy group $\mathrm{Im}(\widetilde{\rho })$ contains $-E_2$, 
since each $B_i$ ($l+1\leq i\leq m$) is of order $4$. 
In this case, replacing $B_i$ by $-B_i$ does not change the isomorphism class of the $T^2$-bundle by Propositions~\ref{cases} (1). Moreover, by applying Proposition~\ref{finite order}, we can change the signs $\ve _i$ ($k+1\leq i\leq l, \; m+1\leq i \leq g$) preserving the isomorphism class of the $T^2$-bundle (recall that $\ve _i$ is fixed to be $-1$ when $l+1\leq i\leq m$). Therefore, by Proposition~\ref{infinite order}, there are just $2^k$ isomorphism classes of $T^2$-bundles $M(\ve _1E_2, B_1, \ldots , \ve _gE_2, B_g)$ corresponding to the choices of $\ve _i$ ($1\leq i\leq k$). 
\item
The case where $m=l$. 
In this case, the order of $B_i$ is not $4$, so each $B_i$ with $k+1\leq i\leq g$ is uniquely determined by the conditions (a) and (c). Moreover, if $-E_2\in \mathrm{Im}(\rho )$, then there exists an integer $i$ with $1\leq i\leq g$ and $\ve_i=-1$. Now we argue the following three cases. 
\begin{enumerate}
\item[(i)]
If there exists an integer $i$ with $1\leq i\leq k$ and $\ve_i=-1$, 
then the isomorphism class depends only on the choices of $\ve_i$ ($1\leq i\leq k$) 
by the same argument as in (1). 
By Proposition~\ref{infinite order}, there are just $2^k-1$ isomorphism classes corresponding to such choices. 
\item[(ii)]
If $\ve_1=\cdots =\ve_k=1$ and there exists an integer $i$ with $k+1\leq i\leq g$ and $\ve _i=-1$, 
then such $T^2$-bundles are all isomorphic to each other by the same argument as in (1). 
Hence, the isomorphism class is unique in this case. 
\item[(iii)]
If $-E_2\not \in \mathrm{Im}(\rho )$, then we have $\ve _1=\cdots =\ve _g=1$. 
In this case, there is no ambiguity other than the choices of the signs of $B_i$ ($1\leq i\leq k$). 
By Proposition~\ref{cases} (2), the $2^k$ choices yield different isomorphism classes. 
Therefore, there are $2^k$ isomorphism classes. 
\end{enumerate}
Thus there are just $2^{k+1}$ isomorphism classes of $T^2$-bundles. 
\end{enumerate}
\end{proof}

\begin{proof}[{\bf Proof of }$\mathrm{{\bf Theorem~\ref{main thm}}}$]
The conditions (1) and (2) in Theorem~\ref{main thm} are equivalent to 
that there exist $f\in \mathcal{M}_g$ and $Q\in SL(2;\Z)$ such that 
$$Q\widetilde{\rho }_2Q^{-1}=\widetilde{\rho }_1\circ f_{\ast }. $$ 
Moreover, the Euler class $e(\xi _1)$ of $M(\ve_1E_2, B_1, \ldots , \ve_gE_2, B_g; m, n)$ lies in 
$$\Z^2/ \left< (B_1-E_2){\bf e}_1,  (B_1-E_2){\bf e}_2, \cdots ,  (B_g-E_2){\bf e}_1,  (B_g-E_2){\bf e}_2 \right>$$ 
if $\ve_1=\cdots =\ve_g=1$, and in 
$$\Z^2/ \left< 2{\bf e}_1, 2{\bf e}_2, (B_1-E_2){\bf e}_1,  (B_1-E_2){\bf e}_2, \cdots ,  (B_g-E_2){\bf e}_1,  (B_g-E_2){\bf e}_2 \right>$$ otherwise. 
Hence the condition (3) in Theorem~\ref{main thm} is equivalent to 
the one that $e(\xi _1)=f^{\ast }e(\xi _2)$ for the above $f\in \mathcal{M}_g$ and $Q\in SL(2;\Z)$. 
Hence, by Proposition~\ref{isom}, the conditions (1), (2), (3) in Theorem~\ref{main thm} give a necessary and sufficient condition for $\xi _1$ and $\xi _2$ to be isomorphic. 
\end{proof}

We obtain the following as a corollary to Theorem~\ref{main thm}. 

\begin{corollary}
Let $\xi _1$ and $\xi _2$ be orientable $T^2$-bundles whose monodromy groups coincide with $SL(2;\Z)$. 
Then $\xi _1$ and $\xi _2$ are bundle isomorphic. 
\end{corollary}
\begin{proof}
We set $B_1=\begin{pmatrix}0&1\\-1&1\end{pmatrix}$, $B_2=\begin{pmatrix}0&1\\-1&0\end{pmatrix}$. 
Then the $SL(2;\Z)$-bundle $$M(E_2, B_1, E_2, B_2, E_2, E_2, \ldots , E_2, E_2)$$ satisfies the condition, 
so it is enough to argue the case where $\xi _1$ is isomorphic to this bundle. Let $\widetilde{\rho }_1$ and $\widetilde{\rho }_2$ denote the monodromies of $\xi _1$ and $\xi _2$, respectively. 
We assume that $p\circ \widetilde{\rho }_2$ is of the normal form in the sense of Theorem~\ref{hom-psl} 
(notice that $p\circ \widetilde{\rho }_1$ is automatically so). 
Then the condition (1) is fulfilled by the assumption $\Im (\widetilde{\rho }_1)=\Im (\widetilde{\rho }_2)=SL(2;\Z)$. 
Since $\Im (p\circ \widetilde{\rho }_1)=\Im (p\circ \widetilde{\rho }_2)$, we may assume that $p\circ \widetilde{\rho }_1=p\circ \widetilde{\rho }_2$ holds by Theorem~\ref{hom-psl}. 
Putting $\rho =p\circ \widetilde{\rho }_1=p\circ \widetilde{\rho }_2$, then $\widetilde{\rho }_1$ and $\widetilde{\rho }_2$ are both lifts of $\rho $ with respect to $p$. 
Since $\Im (\rho )=PSL(2;\Z)\cong \Z_2\ast \Z_3$, 
the integers $k$, $l$, $m$ in Theorem~\ref{hom-psl} are determined as $k=0$, $l=1$, $m=2$. 
Therefore, the $\mathcal{M}_g$-orbit of lifts of $\rho $ is unique by Theorem~\ref{hom-sl}, 
and in particular, the condition (2) is fulfilled. 
Then $SL(2;\Z)$-bundles corresponding to lifts of $\rho $ are all isomorphic to each other. 
In particular, $\xi _2$ is isomorphic to 
$M(E_2, B_1, E_2, B_2, E_2, E_2, \ldots , E_2, E_2; m, n)$ for some $(m,n)\in \Z^2$. 
Finally, we can check the condition (3) as follows. 
Since the four column vectors of the two matrices 
$$B_1-E_2=\begin{pmatrix}-1&1\\-1&0\end{pmatrix}, 
\; B_2-E_2=\begin{pmatrix} -1&1\\-1&-1\end{pmatrix}$$ generates $\Z^2$, 
there are indeed $\x_1$ and $\x_2\in \Z^2$ such that 
$$\begin{pmatrix}0\\0\end{pmatrix}-\begin{pmatrix}m\\n\end{pmatrix}=(B_1-E_2)\x_1+(B_2-E_2)\x_2, $$
which means that the condition (3) is fulfilled. 
Therefore, by Theorem~\ref{main thm}, the $T^2$-bundles $\xi_1$ and $\xi _2$ are isomorphic. 
\end{proof}

\section{$T^2$-bundles with compatible symplectic structures}
Let $\pi \colon M^4\to \Sigma _g$ be an orientable $\Sigma _h$-bundle over a closed orientable surface $\Sigma _g$. A symplectic structure $\omega $ on $M^4$ is said to be compatible with $\pi $ if its restriction to each fiber $\pi ^{-1}(b) \; (b\in \Sigma _g)$ is also a symplectic form. 
Whether a given surface bundle over a surface admits a compatible symplectic structure is determined by the following result. 

\begin{theorem}[Thurston \cite{Th76}]
A $\Sigma _h$-bundle $\pi \colon M^4\to \Sigma _g$ admits a compatible symplectic structure 
if and only if the homology class represented by a fiber $\Sigma _h$ is nonzero in $H_2(M^4; \R)$. 
\end{theorem}

When $h\ne 1$, a compatible symlectic struture always exists since the above condition is automatically fulfilled. When $h=1$, namely, in the case of $T^2$-bundles, the situation is a little more complicated. 
First we consider when $g=0$. In this case, no nontrivial $T^2$-bundle admits a compatible symplectic structure nor even a symplectic form on the total space. For, the second Betti number of the total space of a $T^2$-bundle over $S^2$ is $0$ unless its Euler class is $(0,0)$. When $g=1$, Geiges has given the following answer based on Sakamoto-Fukuhara's classification (Theorem~\ref{SF}).  

\vspace{8pt}

\begin{theorem}[Geiges \cite{Ge92}]~\label{Geiges}
An orientable $T^2$-bundle $\pi \colon M^4\to T^2$ admits a compatible symplectic structure if and only if it is not isomorphic to $$M(E_2, E_2; m, 0) \; (m\ne 0) \;\; \text{nor} \;\; M(E_2, C^k; m, n) \; (n\ne 0), $$ where  $C=\begin{pmatrix} 1&1 \\ 0&1 \end{pmatrix}$ and $k\in \Z$. On the other hand, every orientable $T^2$-bundle over $T^2$ admits a symplectic structure on the total space. 
\end{theorem}

The case where $g\geq 2$ has been settled by Walczak as follows. 

\vspace{8pt}

\begin{theorem}[Walczak \cite{Wa05}]~\label{Walczak}
An orientable $T^2$-bundle $$\pi \colon M(A_1, B_1, \ldots , A_g, B_g; m, n)\to \Sigma _g$$ with $g\geq 2$ admits a compatible symplectic structure if and only if its total space $M^4$ admits a symplectic structure. Moreover, such $T^2$-bundles are classified by their monodromies and Euler classes as follows. 
\begin{enumerate}
\item
$A_i=B_i=E_2$ for all $i$ $(1\leq i\leq g)$ and $(m, n)=(0,0)$. 
\item
The monodromy is nontrivial and there exists a nonzero vector $\x=(x_1, x_2) \in \Z^2$ such that 
$A_i\x =\x, \; B_i\x=\x$ for all $i$ $(1\leq i\leq g)$ and $nx_1-mx_2=0$. 
\item
There is not a nonzero vector $\x \in \Z^2$ such that $A_i\x =\x$ and $B_i\x=\x$ for all $i$. 
\end{enumerate}
\end{theorem}

Thus the problem of determining which surface bundle over a surface admits a symplectic structure has already been solved. However, the statements of Theorems~\ref{Geiges} and \ref{Walczak} on the existence of compatible symplectic structures can be briefly summarized as follows. 
This was pointed out by Yoshihiko Mitsumatsu. 

\vspace{8pt}

\begin{theorem}~\label{Euler}
Let $g$ be a non-negative integer. Then an orientable $T^2$-bundle $$\pi \colon M(A_1, B_1, \ldots , A_g, B_g; m, n)\to \Sigma _g$$ admits a compatible symplectic structure if and only if its Euler class is a torsion. 
\end{theorem}
\begin{proof}
We put $M_0=M(A_1, B_1, \ldots , A_g, B_g)$ and $C=\begin{pmatrix} 1&1 \\ 0&1 \end{pmatrix}$. 
Let $X$ be a $2\times 4g$ matrix consisting of $4g$ column vectors of the following $2g$ matrices; 
$$A_1-E_2, \; B_1-E_2, \; \ldots , \; A_g-E_2, \; B_g-E_2. $$
Then the first Betti number of $M_0$ can be described as $b_1(M_0)=2g+2-\rank (X)$. 
Hence, the three cases $b_1(M_0)=2g+2, \; 2g+1, \; 2g$ can be interpreted to the following three conditions, respectively. 
\begin{enumerate}
\item
$A_1=B_1=\cdots =A_g=B_g=E_2$. 
\item
there exist $Q\in SL(2;\Z)$ and $k_i, l_i\in \Z$ such that 
$A_i=QC^{k_i}Q^{-1}, \; B_i=QC^{l_i}Q^{-1}$. 
\item
otherwise. 
\end{enumerate}
Then the necessary and sufficient condition that Geiges and Walczak gave can be rephrased as follows, 
where $\q_1$ denotes the first column vector of $Q$. 
\begin{enumerate}
\item
and $(m,n)=(0,0)$.
\item
and $\begin{pmatrix}m\\n\end{pmatrix}\in \langle \q_1 \rangle$. 
\item
and $(m, n)$ is arbitrary. 
\end{enumerate}
This condition is equivalent to the one that the Euler class $[(m, n)]\in \Z^2/ \Im (X)$ is a torsion element. 
Indeed, if (1), $\Im (X)=\{ \zv \}$, 
if (2), $$\Im (X)=\langle k_1\q_1, l_1\q_1, \ldots , k_g\q_1, l_g\q_1 \rangle=\langle d \q_1 \rangle ,$$
where $d$ is the greatest common divisor of $k_1, l_1, \ldots , k_g, l_g$, 
and if (3), $\Z^2/\Im (X)$ is a finite group. This completes the proof of the theorem. 
\end{proof}

From the viewpoint of our main theorem, Theorem~\ref{Euler} can be interpreted as follows,  
which seems a natural generalization of Geiges' condition (the former part of Theorem~\ref{Geiges}). 

\vspace{8pt}

\begin{theorem}~\label{compatible}
Let $g$ be a non-negative integer. An orientable $T^2$-bundle $\pi \colon M^4\to \Sigma _g$ admits a compatible symplectic structure if and only if it is not isomorphic to 
$$M(E_2, E_2, \ldots , E_2, E_2; m, 0) \; (m\ne 0) \;\; \text{nor} \;\; M(E_2, C^k, \ldots , E_2, E_2; m, n) \; (n\ne 0), $$ where $C=\begin{pmatrix} 1&1 \\ 0&1 \end{pmatrix}$ and $k\in \Z$. 
\end{theorem}

Combining it with the former part of Theorem~\ref{Geiges} and the latter part of Theorem~\ref{Walczak}, 
we also obtain the following. 

\vspace{8pt}

\begin{theorem}~\label{symplectic}
Let $\pi \colon M^4\to \Sigma _g$ and $C$ be as in Theorem~\ref{compatible}. 
Then $M^4$ admits a symplectic structure if and only if $\pi $ is not isomorphic to either of the following. 
\begin{enumerate}
\item
$M(E_2, E_2, \ldots , E_2, E_2; m, 0) \;\; (g\ne 1, \; m\ne 0) $.
\item
$M(E_2, C^k, \ldots , E_2, E_2; m, n) \;\; (g\geq 2, \; k\in \Z, \; n\ne 0) $.
\end{enumerate}
\end{theorem}

On the other hand, Ue (\cite{Ue90, Ue91, Ue92}) clarified which orientable $T^2$-bundle admits a complex structure on its total space. Comparing Theorem~\ref{symplectic} with his result, we obtain the following theorem. 

\vspace{8pt}

\begin{theorem}~\label{non-Kahler}
Let $\pi \colon M^4\to \Sigma _g$ and $C$ be as in Theorem~\ref{compatible}. Then $M^4$ admits a non-K\"{a}hler symplectic structure if $\pi $ is not isomorphic to either of the following; 
\begin{enumerate}
\item
trivial bundle,  
\item
hyperelliptic bundles $($see \cite{Ue90}, $\text{List }\mathrm{I. \; 1-3 (a))}$,  
\item
$M(E_2, E_2, \ldots , E_2, E_2; m, 0) \;\; (g\ne 1, \; m\ne 0) $, 
\item
$M(E_2, C^k, \ldots , E_2, E_2; m, n) \;\; (g\geq 2, \; k\in \Z, \; n\ne 0) $, 
\item
$M(E_2, B, \ldots , E_2, E_2; m, n) \;\; (g\geq  2, \; \ord(B)=2, 3, 4, 6)$. 
\end{enumerate}
\end{theorem}



By this theorem, we can reconfirm the fact that there exist infinitely many non-K\"{a}hler closed symplectic $4$-manifolds with even first Betti number, which is known as a corollary to Gompf's result \cite{Go95}. This is in a good contrast with the fact that a compact complex surface admits a K\"{a}her metric if and only if its first Betti number is even (\cite{Ko60, Ko63a, Ko63b, Ko64, Ko66, Ko68a, Ko68b, Mi, Si}). 

\begin{remark}
A non-compact complex surface does not necessarily admit a K\"{a}her metric even if its first Betti number is even. Indeed, there exist uncountably many non-K\"{a}hler complex structures on $\R^4$ (\cite{DKZ17}). Moreover, any orientable open $4$-manifold admits both K\"{a}hler structures and non-K\"{a}hler complex structures (\cite{DKZ18}). 
\end{remark}

\section*{Acknowledgements}
The authors would like to thank Professors Masayuki Asaoka, Yoshihiko Mitsumatsu, Hiroki Kodama, Osamu Saeki, Hisaaki Endo, Hajime Sato,  Kokoro Tanaka and Susumu Hirose for their helpful communications. 
Masayuki Asaoka told us the outline of the arguments of Theorem~\ref{hom-free products} and Corollary~\ref{lens}, 
and Yoshihiko Mitsumatsu did that of Theorem~\ref{Euler}. 
Hiroki Kodama gave us helpful advice about Propositions~\ref{F_g} and~\ref{cases}. 
Osamu Saeki, Hisaaki Endo, Hajime Sato, Kokoro Tanaka and Susumu Hirose provided us with numerous references related with this work. 
Finally, the first author is grateful to all the members of the Saturday Topology Seminar for careful checking and active questioning during the early draft stages of this manuscript. 

\bibliographystyle{amsplain}

\end{document}